\documentclass[12pt]{amsart}

\usepackage{multido}
\usepackage{graphicx}
\usepackage{amssymb}
\usepackage{mathabx}
\usepackage{hyperref}

\numberwithin{equation}{section}
\numberwithin{figure}{section}

\allowdisplaybreaks
\DeclareSymbolFont{bbold}{U}{bbold}{m}{n}
\DeclareSymbolFontAlphabet{\mathbbold}{bbold}
\newcommand{\ind}{\mathbbold{1}}

\theoremstyle{plain} \newtheorem{theorem}{Theorem}[section]
\theoremstyle{plain} \newtheorem{proposition}[theorem]{Proposition}
\theoremstyle{plain} \newtheorem{lemma}[theorem]{Lemma}
\theoremstyle{plain} \newtheorem{corollary}[theorem]{Corollary}
\theoremstyle{definition} \newtheorem{definition}[theorem]{Definition}
\theoremstyle{definition} 
\theoremstyle{remark} \newtheorem{remark}[theorem]{Remark}
\theoremstyle{remark} \newtheorem{example}[theorem]{Example}

\newcommand{\bA}{\mathbb{A}}
\newcommand{\bB}{\mathbb{B
}}

\newcommand{\bD}{\mathbb{D}}
\newcommand{\bE}{\mathbb{E}}
\newcommand{\bF}{\mathbb{F}}
\newcommand{\bG}{\mathbb{G}}

\newcommand{\bI}{\mathbb{I}}

\newcommand{\bK}{\mathbb{K}}
\newcommand{\bL}{\mathbb{L}}
\newcommand{\bM}{\mathbb{M}}
\newcommand{\bN}{\mathbb{N}}
\newcommand{\bP}{\mathbb{P}}

\newcommand{\bR}{\mathbb{R}}
\newcommand{\bS}{\mathbb{S}}
\newcommand{\bT}{\mathbb{T}}
\newcommand{\bU}{\mathbb{U}}
\newcommand{\bV}{\mathbb{V}}

\newcommand{\cE}{\mathcal{E}}
\newcommand{\cK}{\mathcal{K}}
\newcommand{\cS}{\mathcal{S}}
\newcommand{\cR}{\mathcal{R}}

\newcommand{\cU}{\mathcal{U}}
\newcommand{\cV}{\mathcal{V}}
\newcommand{\cW}{\mathcal{W}}
\newcommand{\cX}{\mathcal{X}}
\newcommand{\cY}{\mathcal{Y}}
\newcommand{\cZ}{\mathcal{Z}}

\newcommand{\XX}{\mathbf{X}}
\newcommand{\YY}{\mathbf{Y}}
\newcommand{\ZZ}{\mathbf{Z}}
\newcommand{\NN}{\mathbf{N}}
\newcommand{\MM}{\mathbf{M}}

\newcommand{\Deltar}{R}

\newcommand{\fN}{\mathfrak{N}}

\DeclareMathOperator{\diam}{diam}

\newcommand{\dgpr}{d_{\mathrm{GPr}}}
\newcommand{\dpr}{d_{\mathrm{Pr}}}

\renewcommand{\epsilon}{\varepsilon}

\usepackage{fancybox}
\newlength{\querylen}
\begin{document}

\title[Metric measure spaces]{The semigroup of metric measure spaces
  and its infinitely divisible probability measures}

\author[S.N. Evans]{Steven N. Evans}
\address{Department of Statistics \#3860\\
 367 Evans Hall \\
 University of California \\
  Berkeley, CA  94720-3860 \\
   USA} 
\email{evans@stat.berkeley.edu}

\author[I. Molchanov]{Ilya Molchanov}
\address{University of Bern \\
 Institute of Mathematical Statistics and Actuarial Science \\
 Sidlerstrasse 5 \\
 CH-3012 Bern \\
 SWITZERLAND}
\email{ilya.molchanov@stat.unibe.ch}

\thanks{SNE supported in part by NSF grant DMS-09-07630
and NIH grant 1R01GM109454-01. IM supported
  in part by Swiss National Science Foundation grant
  200021-137527.}



\date{\today}

\enlargethispage{5mm}
\begin{abstract}
  A metric measure space is a complete, separable metric space
  equipped with a probability measure that has full support.  Two such
  spaces are equivalent if they are isometric as metric spaces via an
  isometry that maps the probability measure on the first space to the
  probability measure on the second.  The resulting set of equivalence
  classes can be metrized with the Gromov--Prohorov metric of Greven,
  Pfaffelhuber and Winter.  We consider the natural binary operation
  $\boxplus$ on this space that takes two metric measure spaces and
  forms their Cartesian product equipped with the sum of the two
  metrics and the product of the two probability measures.  We show
  that the metric measure spaces equipped with this operation form a
  cancellative, commutative, Polish semigroup with a translation
  invariant metric.  There is an explicit family of
  continuous semicharacters that is extremely useful for, {\em inter
  alia}, establishing that there are no infinitely divisible elements
  and that each element has a unique factorization
  into prime elements.

  We investigate the interaction between the semigroup structure and
  the natural action of the positive real numbers on this space that
  arises from scaling the metric.  
  For example, we show that for any given positive real numbers $a,b,c$
  the trivial space is the only space $\cX$ that satisfies 
  $a \cX \boxplus b \cX = c \cX$. 
  
  We establish that there is no analogue of the law of large numbers: if 
  $\XX_1, \XX_2, \ldots$ is an identically distributed independent 
  sequence of random spaces, then no subsequence of
  $\frac{1}{n} \bigboxplus_{k=1}^n \XX_k$
  converges in distribution unless each $\XX_k$ is almost surely
  equal to the trivial space.
  We characterize the infinitely divisible probability measures and
  the L\'evy processes on this semigroup, characterize the
  stable probability measures and establish a counterpart of the LePage
  representation for the latter class.
\end{abstract}

\maketitle

\section{Introduction}
\label{sec:introduction}

The Cartesian product $G \boxvoid H$ of  two finite graphs $G$ and $H$ with 
respective vertex sets $V(G)$ and $V(H)$ and respective edge sets
$E(G)$ and $E(H)$ is the graph with vertex set 
$V(G \boxvoid H) := V(G) \times V(H)$ 
and edge set 
\[
\begin{split}
E(G \boxvoid H) 
& := 
\{((g',h),(g'',h)) : (g',g'') \in E(G), \, h \in V(H)\} \\
& \quad \cup
\{((g,h'),(g,h'')) : g \in V(G), \, (h',h'') \in E(H)\}. \\
\end{split}
\]
This construction plays a role in many areas of graph theory.
For example, it is shown in \cite{MR0209177} that any connected
finite graph is isomorphic to a Cartesian product of
graphs that are irreducible in the sense that they cannot be represented
as Cartesian products and that this representation is unique up to the
order of the factors (see, also, \cite{MR0209178, MR0274327,
MR0280401, MR904402, MR1787898, MR1192390}).  The study of
the problem of embedding a graph in a Cartesian product
was initiated in \cite{MR776391, MR768159}.  A comprehensive review of
factorization and embedding problems is \cite{MR905281}.

If two connected finite graphs $G$ and $H$ are equipped with the usual shortest
path metrics $r_G$ and $r_H$, then the shortest path metric on the Cartesian
product is given by $r_{G \times H} = r_G \oplus r_H$, where
\[
\begin{split}
(r_G \oplus r_H)((g',h'), (g'',h'')) 
& := r_G(g',g'')+r_H(h',h''), \\
& \quad
(g',h'),\,(g'',h'') \in G \times H. \\
\end{split}
\]
We use the notation $\oplus$ because if we think of the shortest path metric
on a finite graph as a matrix, then the matrix for the
shortest path metric on the Cartesian product of two graphs
is the Kronecker sum of the matrices for the two graphs and
the $\oplus$ notation is commonly used for the Kronecker sum
\cite{MR2807612}.  

It is natural to consider the obvious generalization of this
construction to arbitrary metric spaces and there is a substantial
literature in this direction.  For example, a related binary operation
on metric spaces is considered by Ulam \cite[Problem 77(b)]{MR666400}
who constructs a metric on the Cartesian product of two metric spaces
$(Y, r_Y)$ and $(Z, r_Z)$ via $((y',z'), (y'',z'')) \mapsto
\sqrt{r_Y(y',y'')^2 + r_Z(z',z'')^2}$ and asks whether it is possible
that there could be two nonisometric metric spaces $U$ and $V$ such
that the metrics spaces $U \times U$ and $V \times V$ are isometric.
An example of two such spaces is given in \cite{MR0278262}.  However,
it follows from the results of \cite{MR0256266, MR1230118} that such
an example is not possible if $U$ and $V$ are compact subsets of
a Euclidean space.  

On the other hand, a
classical result of de Rahm \cite{MR0052177}
says that a complete, simply connected, Riemannian manifold has a product
decomposition
$M_0 \times M_1 \times \cdots \times M_k$, where 
the manifold $M_0$ is a Euclidean space (perhaps just a point)
and $M_i$, $i = 1, \ldots , k$, are irreducible Riemannian manifolds 
that each have more than one point and are not isometric to the real line.
By convention, the metric on a product of manifolds is the one appearing in
Ulam's problem.  This last result was extended to the setting of 
geodesic metric spaces of finite dimension in \cite{MR2399098}.

Ulam's problem is closely related to the question of
{\em cancellativity} for this binary operation; that is, if $Y \times
Z'$ and $Y \times Z''$ are isometric, then are $Z'$ and $Z''$
isometric?  This property clearly does not hold in general; for
example, $\ell^2(\bN) \times \ell^2(\bN)$ and $\ell^2(\bN)$ 
(where $\bN := \{0,1,2,\ldots\}$) are
isometric, but $\ell^2(\bN)$ and the trivial metric space are not
isometric. Moreover, an example is
given in \cite{MR1383621} showing that it does not even hold for
arbitrary subsets of $\bR$.  However, we note from \cite{MR1319005}
that there are many compact Hausdorff topological spaces $K$ with the
property that if $L'$ and $L''$ are two compact Hausdorff spaces such
that $K \times L'$ and $K \times L''$ are homeomorphic, then $L'$ and
$L''$ are homeomorphic (see also \cite{MR1843913}).

Returning to the binary operation that combines two metric spaces $(Y,
r_Y)$ and $(Z, r_Z)$ into the metric space $(Y \times Z, r_Y \oplus
r_Z)$, it is shown in \cite{MR1192390} that if a metric space is
isometric to a product of finitely many irreducible metric spaces,
then this factorization is unique up to the order of the factors.
However, there are certainly metric spaces that are not isometric to a
finite product of finitely many irreducible metric spaces and the
study of this binary operation seems to be generally rather difficult.

In this paper we consider a closely-related binary operation on the
class of {\em metric measure spaces}; that is, objects that consist of
a complete, separable metric space $(X, r_X)$ equipped with a
probability measure $\mu_X$ that has full support.  Following
\cite{grom07} (see, also, \cite{ver98, ver03, ver04}), we regard two
such spaces as being equivalent if they are isometric as metric spaces
with an isometry that maps the probability measure on the first space
to the probability measure on the second.  Denote by $\bM$ the set of
such equivalence classes.  With a slight abuse of notation, we will
not distinguish between an equivalence class $\cX \in \bM$ and a
representative triple $(X, r_X, \mu_X)$.

Gromov and Vershik show that a metric measure space $(X, r_X,\mu_X)$
is uniquely determined by the distribution of the infinite random
matrix of distances
\[
(r_X(\xi_i, \xi_j))_{(i,j) \in \bN \times \bN},
\]
where $(\xi_k)_{k \in \bN}$ is an i.i.d. sample of points in $X$ with
common distribution $\mu_X$, and this concise condition for
equivalence makes metric measure spaces considerably easier to study
than metric spaces {\em per se}. A probability measure $Q$ on the
cone $\cR:=\{(r_{ij})_{(i,j)\in\bN\times\bN}\}$ of distance matrices is
the distribution of a distance matrix corresponding to a metric
measure space if and only if it is invariant and ergodic with respect
to action of the infinite symmetric group and for all $\epsilon>0$ there exists
integer $N$ such that
\begin{equation}
\label{Eq:Vershik}
\begin{split}
& Q\Big\{(r_{ij})\in\cR: \lim_{n\to\infty} \frac{\#\{j: 1\leq j\leq n,
  \min_{1\leq i\leq N} r_{ij}<\epsilon \}}{n}>1-\epsilon\Big\} \\
& \quad >1-\epsilon,
\end{split}
\end{equation}
see \cite{ver03}.

We define a binary, associative, commutative operation $\boxplus$
on $\bM$ as follows. 
Given two elements
$\cY = (Y, r_Y, \mu_Y)$ and $\cZ = (Z, r_Z,\mu_Z)$ of
$\bM$, let $\cY \boxplus \cZ$ be  
$\cX = (X, r_X, \mu_X) \in \bM$, where
\begin{itemize}
\item
$X := Y \times Z$, 
\item
$r_X := r_Y \oplus r_Z$, where 
$(r_Y \oplus r_Z)((y',z'), (y'',z'')) =
r_Y(y',y'')+r_Z(z',z'')$ for $(y',z'), (y'',z'') \in Y \times Z$,
\item
$\mu_X : = \mu_Y \otimes \mu_Z$.  
\end{itemize}
The distribution of the random matrix of distances for
$\cY\boxplus\cZ$ is the convolution of the distributions of the random
matrices of distances for $\cY$ and $\cZ$.  The equivalence class
$\cE$ of metric measure spaces that each consist of a single point
with the only possible metric and probability measure on them is the
neutral element for this operation, and so $(\bM, \boxplus)$ is a
commutative semigroup with an identity.  A semigroup with an identity
is sometimes called a {\em monoid}.

\begin{remark}
  We could have chosen other ways to combine the metrics $r_Y$ and
  $r_Z$ to give a metric on $Y \times Z$ that induces the product
  topology and results in a counterpart of $\boxplus$ that is
  commutative and associative.  For example, by analogy with Ulam's
  construction we could have used one of the metrics $((y',z'),
  (y'',z'')) \mapsto
  \left(r_Y(y',y'')^p+r_Z(z',z'')^p\right)^{\frac{1}{p}}$ for $p>1$ or
  the metric $((y',z'), (y'',z'')) \mapsto r_Y(y',y'') \vee
  r_Z(z',z'')$.  We do not investigate these possibilities here.
\end{remark}

We finish this introduction with an overview of the remainder of the paper.

We show in Section~\ref{S:topological} that if we equip $\bM$ with the
Gromov--Prohorov metric $\dgpr$ introduced in \cite{grev:pfaf:win09}
(see Section~\ref{S:metric_summary} for the definition of $\dgpr$),
then the binary operation $\boxplus:\bM \times \bM \to \bM$ is
continuous and the metric $\dgpr$ is translation invariant for the
operation $\boxplus$.  We recall from \cite{grev:pfaf:win09} that
$(\bM, \dgpr)$ is a complete, separable metric space.  Moreover, the
Gromov--Prohorov metric has the property that a sequence of elements of
$\bM$ converges to an element of $\bM$ if and only if the
corresponding sequence of associated random distance matrices
described above converges in distribution to the random distance
matrix associated with the limit.  In Section 2 we also introduce a
partial order $\le$ on $\bM$ by declaring that $\cY\leq\cZ$ if $\cZ =
\cY \boxplus \cX$ for some $\cX\in\bM$ and show for any $\cZ \in \bM$
that the set $\{\cY \in \bM : \cY \le \cZ\}$ is compact.

A semicharacter is a map $\chi: \bM \to [0,1]$ such that $\chi(\cY
\boxplus \cZ) = \chi(\cY) \chi(\cZ)$ for all $\cY, \cZ \in \bM$.  We
introduce a natural family of semicharacters in
Section~\ref{S:semicharacter}.  This family has the property that
$\lim_{n \to \infty} \cX_n = \cX$ for some sequence $(\cX_n)_{n \in
  \bN}$ and element $\cX$ in $\bM$ if and only if $\lim_{n \to \infty}
\chi(\cX_n) = \chi(\cX)$ for all semicharacters $\chi$ in the
family. Using the semicharacters, we characterize the existence of the
limit $\lim_{n \rightarrow \infty} \bigboxplus_{k=0}^n \cX_k$ for some
sequence $(\cX_n)_{n \in \bN}$, and show that if the limit exists,
then $\bigboxplus_{k=0}^n \cX_k'$ converges to the same limit for any
rearrangement $(\cX_n')_{n \in \bN}$ of the sequence.  We also use the
semicharacters to prove that $(\bM,\boxplus)$ is cancellative.

We show in Section~\ref{sec:irred-infin-divis} that the irreducible
elements (that is, those which cannot be decomposed as a nontrivial $\boxplus$
combination of elements of $\bM$)
form a dense, $G_\delta$ subset $\bI\subset\bM$. We give
several examples of irreducible elements; for instance,
all totally geodesic metric measure spaces are
irreducible. Furthermore, there are no nontrivial infinitely
divisible metric measure spaces (an element $\cX \in \bM$ is infinitely
divisible if for every $n \ge 2$ it can be decomposed as the
$\boxplus$-sum of $n$ identical summands).

We establish in Section~\ref{S:arithmetic} that $(\bM,\boxplus)$ is a
Delphic semigroup as studied in \cite{MR0229746,MR0263128}. By appealing
to general results for Delphic semigroups, we confirm that each
metric measure space is either irreducible or has an irreducible
factor and then that any element of $\bM \setminus \{\cE\}$ has a 
representation as either a finite or countable $\boxplus$ combination
of irreducible elements.  We further
show that this representation is unique up to the
order of the ``factors''. The uniqueness does not follow from the
Delphic theory and is based on a result showing that irreducible
elements are prime (an element $\cX \in \bM$ is prime
if $\cX \le \cY \boxplus \cZ$ implies that $\cX \le \cY$
or $\cX \le \cZ$).  

%The unique factorization result has several
%consequences.  For example, it follows readily from it that if $\Phi:
%\bR_+ \to \bM$ is a function such that $\Phi(s+t) = \Phi(s) \boxplus
%\Phi(t)$ for all $0 \le s,t < \infty$, then $\Phi \equiv \cE$. 

In Section~\ref{S:factorization_measure} we investigate the counting
measure on the family $\bI$ of irreducible elements that is obtained
by taking an element of $\bM$ and assigning a unit mass to each
irreducible element (counted according to multiplicity) in its
factorization. We show that this mapping from elements of $\bM$ to
counting measures on $\bM$ concentrated on $\bI$ is measurable in a
natural sense.

Given $\cX \in \bM$ and $a > 0$, we define the rescaled metric measure
space $a \cX := (X, a r_X, \mu_X) \in \bM$.  We show in
Section~\ref{S:scaling} that if $(a \cX) \boxplus (b \cX) = c \cX$ for
some $\cX \in \bM$ and $a,b,c > 0$, then $\cX=\cE$, so the second
distributivity law certainly does not hold for this scaling operation.

We begin the study of random elements of $\bM$ in
Section~\ref{sec:rand-metr-meas} by defining a counterpart of the
usual Laplace transform in which exponential functions are replaced by
semicharacters.  Two random elements of $\bM$ have the same
distribution if and only if their Laplace transforms are equal. A
random element in $\bM$ can be viewed, via its decomposition into irreducibles,
 as a point process on the set
$\bI$ of irreducible elements of $\bM$.

We introduce the appropriate notion of infinitely divisible
random elements of $\bM$ in
Section~\ref{sec:infin-divis-laws} and obtain an analogue
of the classical L\'evy--Hi\u{n}cin--It\^o description of 
infinitely divisible real-valued random variables.  Our approach
to this result is probabilistic and involves constructing for
any infinitely divisible random element a L\'evy process that
at time $1$ has the same distribution as the given random element.
Our setting resembles that of nonnegative infinitely divisible random
variables and so there is no counterpart of a Gaussian component
in this description.  Also, there is no deterministic component because
 the only constant that is infinitely divisible is the trivial space $\cE$.

Using the scaling operation on $\bM$ we define stable random elements
of $\bM$ in Section~\ref{sec:stable-metr-meas}.  We determine how the
L\'evy--Hi\u{n}cin--It\^o description specializes to such random
elements and also verify that there is a counterpart of the LePage
series that represents a stable bounded metric measure space as an
``infinite weighted sum'' of independent identically distributed
random elements in $\bM$ with a suitable independent sequence of
coefficients.

The representation of random elements of $\bM$ as
point processes on the set $\bI$ of irreducible spaces makes it possible 
in Section~\ref{S:thinning} to introduce a thinning operation that takes
an element of $\bM$ and produces another by randomly discarding some of
the irreducible factors. 
L\'evy processes on $\bM$ necessarily have nondecreasing sample
paths with respect to the partial order $\le$, but by combining thinning
with the addition of independent random increments one can produce Markov
processes with sample paths that are not monotone.  Also,
thinning can be used to define a notion of discrete stable random
elements in $\bM$. 

For ease of reference we summarize some facts about the
Gromov--Prohorov metric in Section~\ref{S:metric_summary}. Many of our
arguments can be carried through using alternative metrics on $\bM$ or
its subfamilies such as the $\bD$-metric studied in \cite{MR2237206}.
Lastly, in Section~\ref{S:exponential_inequalities}
we obtain a bound on the Laplace transform of
nonnegative random variables that was useful in Section~\ref{S:semicharacter}.

\section{Topological and order properties}
\label{S:topological}

\begin{lemma}
  \label{lemma:gpr1}
  The operation $\boxplus:\bM \times \bM \to \bM$ is continuous. More
  specifically, if $\cX_i,\cY_i$, $i=1,2$, are elements of $\bM$, then
  \begin{displaymath}
    \dgpr(\cX_1\boxplus \cX_2,\cY_1\boxplus \cY_2)
    \leq \dgpr(\cX_1,\cY_1)+\dgpr(\cX_2,\cY_2)\,.
  \end{displaymath}
\end{lemma}

\begin{proof}
  Let $\phi_{X_i}$ and $\phi_{Y_i}$ be isometries from $X_i$ and $Y_i$
  to a common metric measure space $\cZ_i$, $i=1,2$.  The combined
  function $(\phi_{X_1},\phi_{X_2})$ (resp. $(\phi_{Y_1},\phi_{Y_2})$)
  maps $X_1\times X_2$ (resp. $Y_1\times Y_2$) isometrically into
  $Z_1\times Z_2$. The result now follows from
  Lemma~\ref{L:Prohorov_product}.
\end{proof}

A proof similar to that of Lemma~\ref{lemma:gpr1}
using Lemma~\ref{L:translation_invariance_Prohorov}
establishes the following result.

\begin{lemma}
\label{lemma:gpr2}
The metric $\dgpr$ is translation invariant for the operation $\boxplus$.
That is, if $\cX_1, \cX_2, \cY$ are elements of $\bM$, then
  \begin{displaymath}
    \dgpr(\cX_1\boxplus \cY,\cX_2 \boxplus \cY)
    = \dgpr(\cX_1,\cX_2)\,.
  \end{displaymath}
In particular, 
  \begin{displaymath}
    \dgpr(\cX_1\boxplus \cX_2,  \cX_1)
    = \dgpr(\cX_2,\cE)\,.
  \end{displaymath}
\end{lemma}

\begin{definition}
  Given $\cX = (X, r_X, \mu_X) \in \bM$, write $\diam(\cX)$ for the
  (possibly infinite) diameter of the metric space $X$; that
  is,
  \begin{displaymath}
    \diam(\cX) := \sup\{r_X(x',x'') : x', x'' \in X\}.
  \end{displaymath}
\end{definition}
% And also introduce the essential diameter as
%\[
%\essdiam(\cX) 
%= \sup\{t \ge 0 : \mu_X^{\otimes 2}\{(x_1,x_2) \in X^2:\;
%r_X(x_1,x_2) \ge t\} > 0\}\,,
%\] 
%so that $\essdiam(\cX)$ is the essential supremum of
%$r_X(\xi_1,\xi_2)$ for $(\xi_1,\xi_2)$ distributed according to
%$\mu_X^{\otimes 2}$. The essential diameter is also the essential
%supremum of the \emph{distance distribution}, which is a measure on
%the positive half-line defined by 
%\begin{displaymath}
%  w_W(\cdot)=\mu^{\otimes 2}\{(x_1,x_2):\; d(x_1,x_2)\in \cdot\}\,,
%\end{displaymath}
%see \cite[Def.~2.9]{grev:pfaf:win09}.

The next result is obvious.

\begin{lemma}
  \label{lemma:diam}
% Was previously item a)
The diameter is an additive functional on $(\bM,\boxplus)$;
  that is,
  \begin{displaymath}
    \diam(\cX\boxplus \cY)=\diam(\cX)+\diam(\cY)
  \end{displaymath}
  for all $\cX, \cY \in \bM$. 
\end{lemma}
%\begin{proof}
%  The proof for the diameter is trivial.  It is obvious that
%  $\essdiam(a\cX)=a\essdiam(\cX)$ for all $a>0$, so it suffices to consider
%  $\essdiam(\cX\boxplus\cY)$. By definition,
%  \begin{align*}
%    \essdiam(\cX\boxplus\cY)
%    &= \sup\{t \ge 0 :\; (\mu_X\otimes\mu_Y)^{\otimes
%      2}\{(x_1,y_1,x_2,y_2) \in (X \times Y)^2:\;\\
%    &\qquad\qquad r_X(x_1,x_2)+r_Y(y_1,y_2) \ge t\} > 0\}\\
%    &= \sup\{t \ge 0 :\; \mu_X^{\otimes 2}\otimes\mu_Y^{\otimes
%      2}\{(x_1,x_2,y_1,y_2) \in X^2 \times Y^2:\;\\
%    &\qquad\qquad r_X(x_1,x_2)+r_Y(y_1,y_2) \ge t\} > 0\}\\
%    &=\essdiam(\cX)+\essdiam(\cY)\,,
%  \end{align*}
%  which is seen by proving two inequalities. In particular, we used
%  the fact that if $\mu_X^{\otimes 2}\otimes\mu_Y^{\otimes 2}(A)>0$,
%  then there are measurable sets $A_X\subset X^2$ and $A_Y\subset Y^2$
%  of positive measures $\mu_X^{\otimes 2}$ and $\mu_Y^{\otimes 2}$
%  such that $A_X\times A_Y\subset A$. 
%\end{proof}

\begin{remark}
  The function $\diam$ is not continuous even on the family $\bK$ of
  \emph{compact} metric measure spaces.  For example, let $\cX_n =
  (\{0,1\}, r, \mu_n)$, where $r(0,1) = 1$, $\mu_n\{0\} = 1 -
  \frac{1}{n}$ and $\mu_n\{1\} = \frac{1}{n}$.  Then, $\cX_n$
  converges to the trivial space $\cE$, whereas $\diam(\cX_n) = 1 \not
  \to 0 = \diam(\cE)$.
\end{remark}

\begin{lemma}
  \label{L:Diam-lsc}
  The function $\diam$ is lower semicontinuous on $\bM$. That is, if
  the sequence $\cX_n\to\cX$ in $\bM$ as $n \to \infty$, then
  $\diam(\cX)\le \liminf_{n \to \infty} \diam(\cX_n)$.
\end{lemma}
\begin{proof}
  Suppose that the sequence $\cX_n$ converges to $\cX$,
  $(\xi_k^{(n)})_{k\in\bN}$ are i.i.d. in $X_n$ with the common
  distribution $\mu_{X_n}$, and $(\xi_k)_{k\in\bN}$ are i.i.d. in $X$
  with the common distribution $\mu_X$. Observe for any $k$ that
  $\max_{1\le i<j\leq k}(r_{X_n}(\xi_i^{(n)},\xi_j^{(n)})$ converges
  in distribution to $\max_{1\le i<j\leq k}(r_X(\xi_{i},\xi_{j}))$.
  It suffices to note that 
  $\max_{1\le i<j\leq
    k}(r_{X_n}(\xi_i^{(n)},\xi_j^{(n)}))$ 
    is increasing in $k$ and
  converges almost surely to $\diam(\cX_n)$ as $k\to\infty$ and that
  $\max_{1\le i<j\leq k}(r_{X}(\xi_i,\xi_j))$ is increasing in $k$ and
  converges almost surely to $\diam(\cX)$ as $k\to\infty$.
\end{proof}

\begin{definition} 
  Define a partial order $\le$ on $\bM$ by setting $\cY\leq\cZ$ if
  $\cZ = \cY \boxplus \cX$ for some $\cX\in\bM$.
\end{definition}

The symmetry and transitivity of $\le$ is obvious.  The antisymmetry is 
apparent from Lemma~\ref{L:GPr-inequality} below. This partial order
is the dual of the {\em Green} or {\em divisibility} order 
(see \cite[Section I.4.1]{MR2017849}).
The identity $\cE$ is the unique minimal element. 

\begin{lemma}
  \label{L:GPr-inequality}
  If $\cX\leq\cY\leq\cZ$, then $\dgpr(\cX,\cY)\leq\dgpr(\cX,\cZ)$.
\end{lemma}

\begin{proof}
  In view of Proposition~\ref{P:cancellation}(a) and
  Lemma~\ref{lemma:gpr2}, it suffices to assume that $\cX=\cE$. If
  $\cZ=\cY\boxplus\cV$, then \eqref{Eq:DGPr-to-neutral} yields that
  \begin{align*}
    \dgpr(\cZ,\cE)&=
    \inf_{y\in\cY,v\in\cV}
    \inf\{\epsilon>0: \mu_Y\otimes\mu_V\{(y',v'):\\
    &\qquad\qquad\qquad r_Y(y,y')+r_V(v,v')\geq \epsilon\}\leq\epsilon\}\\
    &\geq \inf_{y\in\cY,v\in\cV}
    \inf\{\epsilon>0: \mu_Y\otimes\mu_V\{(y',v'):
    r_X(y,y')\geq \epsilon\}\leq\epsilon\}\\
    &=\dgpr(\cY,\cE)\,.
  \end{align*}
\end{proof}

An element of a semigroup with an identity is a {\em unit} if it has
an inverse and a semigroup with an identity is said to be
\emph{reduced} if the only unit is the identity (see
\cite[Section~1]{MR1503427}.  

\begin{corollary}
The semigroup $(\bM, \boxplus)$ is reduced.
\end{corollary}

\begin{proof}
Suppose that $\cE = \cX \boxplus \cY$, then $\cE \le \cX \le \cE$
and the antisymmetry of the partial order $\le$ gives that $\cE = \cX = \cY$.
\end{proof}

\begin{lemma}
  \label{L:below_Z_compact}
  \begin{itemize}
  \item[a)]
  For any compact set $\bS \subset \bM$, the set 
  $\bigcup_{\cZ \in \bS} \{\cY \in \bM : \cY \le \cZ\}$ 
  is compact.
  \item[b)]
  For any compact set $\bS \subset \bM$, the set 
  $\{(\cY,\cZ) \in \bM^2 : \cZ \in \bS, \, \cY \le \cZ\}$ 
  is compact.
  \item[c)]
  The map $K$ from $\bM$ to the compact subsets of $\bM$
  defined by $K(\cX):=\{\cY \in \bM:\; \cY\leq\cX\}$ is upper semicontinuous.
  That is, if $F \subseteq \bM$ is closed, then 
  $\{\cX \in \bM : K(\cX) \cap F \ne \emptyset\}$ is closed.
  Equivalently, if $\cX_n\to\cX$, and 
  $\cY_n \in K(\cX_n)$ converges to $\cY$, then $\cY \in K(\cX)$. 
  \end{itemize}
\end{lemma}

\begin{proof}
  (a) We first show that 
  $\bigcup_{\cZ \in \bS} \{\cY \in \bM : \cY \le \cZ\}$ is pre-compact.
  Given $\epsilon > 0$, we know from \cite[Theorem~2]{grev:pfaf:win09}
  that there exist $K>0$ and $\delta>0$ such that for all $\cZ \in \bS$
  \[
  \mu_Z \otimes \mu_Z\{(z',z'') \in Z \times Z: r_Z(z',z'') > K\} \le \epsilon
  \]
  and
  \[
  \mu_Z\{z' \in Z : \mu_Z\{z'' \in Z : r_Z(z',z'') < \epsilon\} \le \delta\} 
  \le \epsilon.
  \]
  If $\cY \le \cZ$ for some $\cZ \in \bS$, 
  then, by definition, there is a $\cW \in \bM$ such that $\cZ = \cY
  \boxplus \cW$,  and so
  \[
  \begin{split}
    & \mu_Y \otimes \mu_Y\{(y',y'') \in Y \times Y: r_Y(y',y'') > K\} \\
    & \quad \le
    (\mu_Y \otimes \mu_Y) \otimes (\mu_W \otimes \mu_W) 
    \{((y',y''),(w',w'')) \in (Y \times Y) \times (W \times W) : \\
    & \qquad\qquad r_Y(y',y'') + r_W(w',w'') > K\} \\
    & \quad =
    \mu_Z \otimes \mu_Z\{(z',z'') \in Z \times Z: r_Z(z',z'') > K\} \\
    & \quad \le \epsilon.
  \end{split}
  \]
  Similarly,
  \[
  \begin{split}
    & \mu_Y\{y' \in Y : \mu_Y\{y'' \in Y : r_Y(y',y'') < \epsilon\} \le \delta\} \\
    & \quad =\mu_Y\otimes \mu_W\{(y',w')\in Y\times W: \\
    &\qquad\qquad \mu_y\otimes\mu_W\{(y'',w'') \in Y \times W : r_Y(y',y'') < \epsilon\} \le \delta\} \\
    & \quad \le
    \mu_Y \otimes \mu_W\{(y',w') \in Y \times W
    : \mu_Y\otimes \mu_W\{(y'',w'') \in Y \times W 
    : \\
    & \qquad\qquad r_Y(y',y'') + r_W(w',w'') < \epsilon\} \le \delta\} \\
    & \quad =
    \mu_Z\{z' \in Z : \mu_Z\{z'' \in Z : r_Z(z',z'') < \epsilon\} \le \delta\} \\
    & \quad \le \epsilon\,. \\
  \end{split}
  \]
  It follows from \cite[Theorem~2]{grev:pfaf:win09} that 
  $\bigcup_{\cZ \in \bS} \{\cY \in \bM : \cY \le \cZ\}$ is pre-compact.

  We now show that $\bigcup_{\cZ \in \bS} \{\cY \in \bM : \cY \le \cZ\}$ 
  is closed, and hence compact.
  Suppose now that $(\cY_n)_{n \in \bN}$ is a sequence in 
  $\bigcup_{\cZ \in \bS} \{\cY \in \bM : \cY \le \cZ\}$ that converges to a limit $\cY_\infty$.  For each $n \in \bN$ we can find 
  $\cZ_n \in \bS$ and 
  $\cW_n \in \bigcup_{\cZ \in \bS} \{\cY \in \bM : \cY \le \cZ\}$
  such that $\cZ_n = \cY_n \boxplus \cW_n$.  
  From the above we can find a subsequence $(n(k))_{k \in
    \bN}$, $\cZ_\infty \in \bS$ and $\cW_\infty \in \bM$ such that 
    $\lim_{k \to \infty} \cZ_{n(k)} = \cZ_\infty$ and 
    $\lim_{k \to \infty} \cW_{n(k)} = \cW_\infty$.  
    By the continuity of the semigroup
  operation established in Lemma~\ref{lemma:gpr1}, 
  \begin{displaymath}
    \cY_\infty \boxplus \cW_\infty 
    = \lim_{k \to \infty}
    (\cY_{n(k)} \boxplus \cW_{n(k)}) 
    = \lim_{k \to \infty}
    \cZ_{n(k)}
    = \cZ_\infty,
  \end{displaymath}
  which implies that $\cY_\infty \le \cZ_\infty \in \bS$ 
  (and also $\cW_\infty \le \cZ_\infty \in \bS$).
  Therefore, $\bigcup_{\cZ \in \bS} \{\cY \in \bM : \cY \le \cZ\}$ 
  is closed and hence compact.
  
\noindent
(b) Because $\{(\cY,\cZ) \in \bM^2 : \cZ \in \bS, \, \cY \le \cZ\}$
is a subset of the compact set 
$(\bigcup_{\cZ \in \bS} \{\cY \in \bM : \cY \le \cZ\}) \times \bS$,
it suffices to show that the former set is closed, but this follows from
an argument similar to that which completed the proof of part (a).

\noindent
(c) This is immediate from (b).
\end{proof}

% \begin{lemma}
%   \label{L:below-compact}
%   If $\cZ$ is a compact metric measure space and $\cX\leq\cZ$, then
%   $\cX$ is also compact.
% \end{lemma}
% \begin{proof}
%   If $\cZ=\cX\boxplus\cY$ and $\{x_n\}_{n\in\bN}$ is a sequence from
%   $X$, then $\{(x_n,y_n\}_{n\in\bN}$ is a sequence from $Z$ for any
%   sequence $\{y_n\}_{n\in\bN}$ from $Y$. By compactness, $z_{n(k)}$
%   converges and so is fundamental. Therefore, $\{x_{n(k)}\}_{k\in\bN}$
%   is also fundamental and so converges in view of the completeness of
%   $(X,r_X)$.
% \end{proof}

% \begin{remark}
%   Any partially ordered space can be endowed with a corresponding
%   Scott topology generated by the order, see \cite{MR1975381}. In
%   particular, the Scott topology on $(\bM,\le)$ is much weaker than
%   the one induced by the Gromov--Prohorov metric.
% \end{remark}

\section{Semicharacters}
\label{S:semicharacter}

Following the standard terminology in semigroup theory, a
{\em semicharacter} is a map $\chi: \bM \to [0,1]$ such that
$\chi(\cY \boxplus \cZ) = \chi(\cY) \chi(\cZ)$ for all $\cY, \cZ \in \bM$.

\begin{definition}
  Denote by $\bA$ the set consisting of the
  empty set and the arrays $A = (a_{ij})_{1 \le i < j \le
    n} \in \bR_+^{\binom{n}{2}}$ for $n\geq 2$.  For each $A \in
  \bA$ define a semicharacter $\chi_A$ by setting $\chi_\emptyset \equiv 1$
  and
\begin{equation}
  \label{eq:semicharacter}
  \chi_A((X, r_X, \mu_X)) 
  :=
  \int_{X^n} \exp\left(- \sum_{1 \le i < j \le n} a_{ij} r_X(x_i,x_j)\right) \,
  \mu_X^{\otimes n}(dx)
\end{equation}
if $A \ne \emptyset$.
Note that $\chi_A(\cX) > 0$ for all $A \in \bA$ and $\cX \in \bM$.
We often need the particular semicharacter 
\begin{equation}
  \label{eq:chi-1}
  \chi_1(\cX) := \int_{X^2} \exp(-r_X(x_1, x_2)) \, \mu_X^{\otimes
  2}(dx)
\end{equation}
defined by taking as $A\in\bA$ an array with the single element
$1$. 
\end{definition}

As we recalled in the Introduction, a metric measure space 
$(X, r_X,\mu_X)$ is uniquely determined by the distribution of the infinite
random matrix of distances $(r_X(\xi_i, \xi_j))_{(i,j) \in \bN \times \bN}$,
where $(\xi_k)_{k \in \bN}$ is an i.i.d. sample of points in $X$ with
common distribution $\mu_X$. 
%Moreover, Vershik gives necessary and
%sufficient conditions on an $\bN \times \bN$ matrix in order for it to
%be constructed in this way from a general metric measure space and
%from a compact metric measure space.
The next lemma follows immediately from this observation
and the unicity of Laplace transforms.

\begin{lemma}
  \label{L:semicharacters_separate}
  \label{L:semicharacters_inequality}
  \begin{itemize}
  \item[a)] Two elements $\cX, \cY \in \bM$ are equal if and only if
    $\chi_A(\cX) = \chi_A(\cY)$ for all $A \in \bA$.
  \item[b)] If $\cY\leq\cX$, then $\chi_A(\cX)\geq \chi_A(\cY)$ for
    all $A\in\bA$.
  \end{itemize}
\end{lemma}

% Note that the ordering of the values of semicharacters does not imply
% that the spaces are ordered, which is seen by looking at the spaces 
% $[0,1]$ and $[0,2]$.

\begin{remark}
\label{R:semicharacter_semigroup}
Note that if $A' \in \bR_+^{\binom{n'}{2}}$
and $A'' \in \bR_+^{\binom{n''}{2}}$, then 
$\chi_{A'}\chi_{A''}=\chi_A$, where $A \in\bR_+^{\binom{n'+n''}{2}}$
is given by
\[
a_{ij} 
= 
\begin{cases}
a'_{ij},& \quad 1 \le i < j \le n' \\
a''_{i-n',j-n'},& \quad n'+1 \le i < j < n'+n''.
\end{cases}
\]
It follows that $\{\chi_A : A \in \bA\}$ is a semigroup with
identity $\chi_\emptyset \equiv 1$.
\end{remark}

\begin{remark}
  Not all semicharacters of $\bM$ are of the form $\chi_A$ for some $A
  \in \bA$.  For example, if $A \in \bA$ and $\beta > 0$, then
  $\cX \mapsto \chi_A(\cX)^\beta$ is a (continuous)
  semicharacter.  If $X$ has two points,
  say $0$ and $1$, that are distance $r$ apart and 
  $\mu_X(\{0\}) = (1-p)$ and $\mu_X(\{1\}) = p$ for some $0 < p < 1$, 
  then taking $A$ to be the array with the single element $a$ we have
  $\chi_A(\cX) = (1-p)^2 + p^2 + 2 p ( 1-p) \exp(-a r)$ and it is not
  hard to see from considering just $\cX$ of this special type that for 
  $\beta \ne 1$ the semicharacter $\chi_A^\beta$ is not of the form $\chi_{A'}$
  for some other $A \in \bA$.
  
  It follows from Lemma~\ref{lemma:diam} that $\cX \mapsto
  \exp(-\diam(\cX))$ is a (discontinuous) semicharacter on
  $\bM$. Also, if $A \in \bA$ and $b > 0$, then
  \[
  \left(
  \int_{X^n} \exp\left( \sum_{1 \le i < j \le n} a_{ij} r_X(x_i,x_j)\right) \,
  \mu_X^{\otimes n}(dx)
  \right)^{-b}
  \]
  is a (discontinuous) semicharacter. These last two examples are connected by
  the observation that
\[
\exp(-\diam(\cX)) = 
  \lim_{t \to \infty} \left(\int_{X^2} \exp\left(t \, r_X(x_1,x_2)\right) \,
  \mu_X^{\otimes 2}(dx)\right)^{-\frac{1}{t}}. 
\]
\end{remark}

\begin{lemma}
\label{L:Gromov_weak}
A sequence $(\cX_n)_{n \in \bN}$ converges to $\cX
  \in \bM$ if and only if $\lim_{n \to \infty} \chi_A(\cX_n)
  = \chi_A(\cX)$ for all $A \in \bA$.
\end{lemma}

\begin{proof}
For $n \in \bN$,
let $(\xi_k^{(n)})_{k \in \bN}$ be an i.i.d. sequence of $X_n$-valued
random variables with common distribution $\mu_{X_n}$, and
let $(\xi_k)_{k \in \bN}$ be an i.i.d. sequence of $X$-valued
random variables with common distribution $\mu_X$.  It follows from
\cite[Theorem 5]{grev:pfaf:win09} that $\cX_n$ converges to $\cX$
if and only if the distribution of
$(r_{X_n}(\xi_i^{(n)}, \xi_j^{(n)}))_{1 \le i < j \le m}$
converges to that of 
$(r_X(\xi_i, \xi_j))_{1 \le i < j \le m}$ for all $m \in \bN$.
The result is now a consequence of the equivalence between the weak convergence
of probability measures on $\bR_+^{\binom{m}{2}}$ and the convergence
of their Laplace transforms.
\end{proof}

In the usual terminology of semigroup theory, part (a) of the following
result says that the semigroup $(\bM,\boxplus)$ is {\em cancellative} (see
\cite[Section II.1.1]{MR2017849}).

\begin{proposition}
  \label{P:cancellation}
  \begin{itemize}
  \item[a)] Suppose that $\cY, \cZ', \cZ'' \in \bM$ satisfy $\cY
    \boxplus \cZ'= \cY \boxplus \cZ''$, then $\cZ' = \cZ''$. If $\cY
    \boxplus \cZ'\leq \cY \boxplus \cZ''$, then $\cZ'\leq \cZ''$. 
  \item[b)] Consider sequences $(\cX_n)_{n \in \bN}$ and $(\cY_n)_{n
      \in \bN}$ in $\bM$.  Set $\cZ_n := \cX_n \boxplus \cY_n$.
    Suppose that $\cX := \lim_{n \to \infty} \cX_n$ and $\cZ :=
    \lim_{n \to \infty} \cZ_n$ exist.  Then, $\cY := \lim_{n \to
      \infty} \cY_n$ exists and $\cZ = \cX \boxplus \cY$.
  \end{itemize}
\end{proposition}

\begin{proof}
  a) For each semicharacter $\chi_A$, $A \in \bA$, we have
  $\chi_A(\cY) \chi_A(\cZ') = \chi_A(\cX) = \chi_A(\cY) \chi_A(\cZ'')$
  and so $\chi_A(\cZ') = \chi_A(Z'')$, which implies that $\cZ' =
  \cZ''$. In case of the inequality,
  $\cY\boxplus\cZ'\boxplus\cW=\cY\boxplus\cZ''$, so that
  $\cZ'\boxplus\cW=\cZ''$ and hence $\cZ'\leq\cZ''$.
  
  \noindent
  b) By Lemma~\ref{L:below_Z_compact}(a), the sequence $(\cY_n)_{n \in
    \bN}$ is pre-compact.  Any subsequential limit $\cY_\infty$ will
  satisfy $\cZ = \cX \boxplus \cY_\infty$. It follows from part (a)
  that $\cY := \lim_{n \to \infty} \cY_n$ exists and $\cZ = \cX
  \boxplus \cY$ in view of Lemma~\ref{L:semicharacters_separate}(a).
\end{proof}

\begin{remark}
\label{R:group_embedding}
It is a consequence of Proposition~\ref{P:cancellation}(a) and the discussion
in Section 1.10 of \cite{MR0132791} that the semigroup $(\bM, \boxplus)$ can
be embedded into a group $\bG$ as follows.  Equip $\bM \times \bM$ 
with the equivalence relation $\equiv$ defined by $(\cW,\cX) \equiv (\cY, \cZ)$
if $\cW \boxplus \cZ = \cX \boxplus \cY$.  It is not hard to see that 
$\equiv$ is indeed an equivalence relation, the only property 
that is not completely obvious is transitivity. However, if
$(\cU, \cV) \equiv (\cW,\cX)$ and $(\cW,\cX) \equiv (\cY, \cZ)$, then,
by definition,
$\cU \boxplus \cX = \cV \boxplus \cW$
and
$\cW \boxplus \cZ = \cX \boxplus \cY$
so that
\[
\begin{split}
&(\cU \boxplus \cZ) \boxplus (\cX \boxplus \cW)
=
(\cU \boxplus \cX) \boxplus (\cW \boxplus \cZ) \\
& \quad =
(\cV \boxplus \cW) \boxplus (\cX \boxplus \cY)
= 
(\cV \boxplus \cY) \boxplus (\cX \boxplus \cW), \\
\end{split}
\]
from which we see that
$\cU \boxplus \cZ = \cV \boxplus \cY$
and hence
$(\cU, \cV) \equiv (\cY, \cZ)$.
The elements of the group $\bG$ are the equivalence classes
for this relation.  We write $\boxplus$ for the binary operation on 
$\bG$ and define it to be the operation that takes the equivalence
classes of $(\cW, \cX)$ and $(\cY, \cZ)$ to the equivalence class
of $(\cW \boxplus \cY, \cX \boxplus \cZ)$.  It is clear that
this operation is well-defined, associative and commutative.
The identity element is the equivalence class of $(\cE, \cE)$ and the
inverse of the equivalence class of $(\cY, \cZ)$ is the
equivalence class of $(\cZ, \cY)$.
\end{remark}

It will be convenient for us to have
various ways of measuring how far a metric measure space $\cX$ is
from the trivial space $\cE$.  The most obvious such measure is
simply the Gromov--Prohorov distance $\dgpr(\cX,\cE)$.  
Note from Lemma~\ref{lemma:gpr1} that 
$\dgpr(\cX_1 \boxplus \cX_2,\cE) 
\leq \dgpr(\cX_1,\cE) + \dgpr(\cX_2,\cE)$ for $\cX_1, \cX_2 \in \bM$.
It follows from
Lemma~\ref{L:Gromov_weak} that a
sequence $(\cX_n)_{n \in \bN}$ is such that
$\dgpr(\cX_n,\cE) \rightarrow 0$
if and only if $\chi_A(\cX_n) \rightarrow 1$ for all $A \in \bA$
and so $D_A(\cX):=-\log \chi_A(\cX)$ is also a measure of how far
$\cX$ is from $\cE$. Observe that $D_A(\cX_1 \boxplus \cX_2) = D_A(\cX_1)
+ D_A(\cX_2)$ for $\cX_1, \cX_2 \in \bM$. To simplify notation, we set
\begin{equation}
  \label{Eq:D_1}
  D(\cX):=-\log\chi_1(\cX). 
\end{equation}
It is a consequence of Lemma~\ref{L:chi1_chiA} below that 
$\dgpr(\cX_n,\cE) \rightarrow 0$ if and only if $D(\cX_n) \to 0$.

The equivalence between convergence in the Gromov--Prohorov distance
and convergence in distribution of the corresponding random distance
matrices implies that if we set
\begin{equation}
  \label{Eq:Delta}
  \Deltar(\cX):=\int_{X^2} (r_X(x_1,x_2)\wedge 1) \, \mu_X^{\otimes 2}(dx),
\end{equation}
then $\dgpr(\cX_n,\cE) \rightarrow 0$ if and only if 
$\Deltar(\cX_n) \rightarrow 0$.
It is clear that $\Deltar(\cX_1 \boxplus \cX_2) 
\le \Deltar(\cX_1) + \Deltar(\cX_2)$ for $\cX_1, \cX_2 \in \bM$.
One last quantity that useful for measuring how far bounded
metric measure spaces are from $\cE$ is the diameter.  
Recall from Lemma~\ref{lemma:diam} that 
$\diam(\cX_1 \boxplus \cX_2) 
= \diam(\cX_1) + \diam(\cX_2)$ for $\cX_1, \cX_2 \in \bM$.
The following result establishes a number of relationships between
these various objects.

\begin{lemma}
  \label{L:chi1_chiA}
%Contains material that was previously in Lemma~\ref{L:R_inequalities}
\begin{itemize}
\item[a)]
  For each $A\in\bA$, there exist constants $a \ge b > 0$ such that, for all
  $\cX\in\bM$, 
  \begin{displaymath}
    \chi_1(\cX)^a \le \chi_A(\cX) \le \chi_1(\cX)^b
  \end{displaymath}
 and hence
\[
b D(\cX) \le D_A(\cX) \le a D(\cX).
\]
\item[b)]
For each $\cX\in\bM$, 
\[
\frac{1}{4}\Deltar(\cX)\leq \dgpr(\cX,\cE)\leq \sqrt{\Deltar(\cX)}.
\]
\item[c)]
There exist constants $C > c > 0$ such that for each $\cX\in\bM$
\[
c(D(\cX) \wedge 1) \le \Deltar(\cX) \le C(D(\cX) \wedge 1).
\]
\item[d)]
For each $\cX \in \bM$,
$\dgpr(\cX,\cE) \le \diam(\cX)$,
$D(\cX) \le \diam(\cX)$
and
$\Deltar(\cX) \le \diam(\cX)$.
\end{itemize}
\end{lemma}
\begin{proof}
Consider the the first inequality in part (a)
for $A\in\bA \cap \bR^{\binom{n}{2}}$.
The triangle
  inequality yields that
  \begin{displaymath}
    \sum_{1\leq i<j\leq n} a_{ij}r_X(x_i,x_j)
    \leq c\sum_{i=2}^n r_X(x_1,x_i)
  \end{displaymath}
  for a certain constant $c$. Therefore,
  \begin{align*}
    \chi_A(\cX)&\geq \int_X \left( \int_X \exp(-cr_X(x,y))\mu_X(dy)\right)^{n-1} \, \mu_X(dx)\\
    &\geq \left( \int_{X^2} \exp(-cr_X(x,y)) \, \mu_X(dx) \, \mu_X(dy)\right)^{n-1}\\
    &\geq (\chi_1(\cX))^{(c\wedge 1)(n-1)}\,,
  \end{align*}
  where the two last inequalities follow from Jensen's inequality.

Regarding the second inequality in part (a), there exist $1 \le i' < j' \le n$
such that $0 < a_{i'j'} =: \alpha$. Because
$\sum_{1\leq i<j\leq n} a_{ij}r_X(x_i,x_j) \ge \alpha r_X(x_{i'},x_{j'})$,
we have $\chi_A(\cX) \le \chi_\alpha(\cX)$. If $\alpha \ge 1$, then
$\chi_\alpha(\cX) \le \chi_1(\cX)$, whereas if $\alpha < 1$, then
$\chi_\alpha(\cX) \le \chi_1(\cX)^\alpha$ by Jensen's inequality.  Therefore,
$\chi_A(\cX) \le \chi_\alpha(\cX) \le \chi_1(\cX)^{\alpha \wedge 1}$.

For the first inequality in (b), we begin by recalling
\eqref{Eq:DGPr-to-neutral} which says that
\[
  \dgpr(\cX,\cE)=\inf_{x\in X} \inf\{\epsilon>0:\; 
  \mu_X\{y\in X:\; r_X(x,y) \ge \epsilon\}\leq\epsilon\}\,.
\]
Suppose that
$\dgpr(\cX,\cE) < \gamma$ where $0 < \gamma \le 1$. 
There is then an $x \in X$ such that
$\mu_X\{y\in X:\; r_X(x,y)\geq \gamma\}\leq\gamma$.  Hence, by the triangle
inequality
\[
\begin{split}
\Deltar(\cX) & =\int_{X^2} (r_X(y_1,y_2)\wedge 1) \, \mu_X^{\otimes 2}(dy) \\
& \le
\int_{X^2} ([r_X(x,y_1) + r_X(x,y_2)]\wedge 1) \, \mu_X^{\otimes 2}(dy) \\
& \le
2 \int_{X} (r_X(x,y) \wedge 1) \, \mu_X(dy) \\
& \le
2\left[
\gamma \mu_X\{y\in X:\; r_X(x,y) < \gamma\}
+
\mu_X\{y\in X:\; r_X(x,y) \ge \gamma\} \right]\\
& \le
4 \gamma,\\
\end{split}
\]
and the inequality follows.

Turning to the second inequality in part (b), suppose that
$\Deltar(\cX) < \gamma$ where $0 < \gamma \le 1$.  There must then
be an $x \in \cX$ for which 
$\int_{X} (r_X(x,y) \wedge 1) \, \mu_X(dy) < \gamma$ and hence
$\epsilon \mu_X\{y\in X:\; r_X(x,y) \ge \epsilon\} < \gamma$ for
$0 < \epsilon \le 1$.  Take $\epsilon = \sqrt{\gamma}$ to see that
$\mu_X\{y\in X:\; r_X(x,y) \ge \sqrt{\gamma}\} < \sqrt{\gamma}$,
as required.

Part (c) is immediate from Lemma~\ref{L:exponential_inequalities}. 
and part (d) is obvious.
\end{proof}

%If $\cX\leq\cY$, i.e. $\cY=\cX\boxplus\cV$, then 
%\begin{align*}
%  \Deltar(\cY)&=\int_{X^2\times V^2} 
%  ((r_X(x_1,x_2)+r_V(v_1,v_2))\wedge 1)\mu_X^{\otimes
%    2}(dx)\mu_V^{\otimes 2}(dv)\\
%  &\geq \int_{X^2\times V^2} 
%  (r_X(x_1,x_2)\wedge 1)\mu_X^{\otimes
%    2}(dx)\mu_V^{\otimes 2}(dv)\\
%  &=\Deltar(\cX).
%\end{align*}

\begin{proposition}
\label{P:convergence_facts}
\begin{itemize}
\item[a)]
The sequence $(\bigboxplus_{k=0}^n \cX_k)_{n \in \bN}$ converges in
  $\bM$ if and only if 
$\lim_{m,n \to \infty, \, m<n} \bigboxplus_{k=m+1}^n X_k = \cE$.
%was \label{L:Gromov_weak} b)
\item[b)] 
The sequence $(\bigboxplus_{k=0}^n \cX_k)_{n \in \bN}$ converges in
$\bM$ if and only if $\sum_{k \in \bN} D(\cX_k)<\infty$ or, equivalently,
if and only if $\sum_{k \in \bN} \Deltar(\cX_k)<\infty$.
  %was \label{C:monotone_seqs_converge} a)
\item[c)] 
The sequence $(\bigboxplus_{k=0}^n \cX_k)_{n \in \bN}$ converges in
$\bM$ if and only if there exists $\cZ \in \bM$
such that $\bigboxplus_{k=0}^n \cX_k \le \cZ$ for all $n \in
\bN$, in which case
$\lim_{n\to\infty}\bigboxplus_{k=0}^n \cX_k\leq \cZ$.  
%was \label{C:monotone_seqs_converge} b)
\item[d)] Suppose that $(\cY_n)_{n \in \bN}$ is a sequence in $\bM$ such that
$\cY_0 \ge \cY_1 \ge \cdots$.  Then,
$\lim_{n \to \infty} \cY_n$ exists.
% was \label{P:sum_convergence} a)
\item[e)] 
Suppose that $(\cX_n)_{n \in \bN}$ is a sequence such that
  $\lim_{n \to \infty} \bigboxplus_{k=0}^n \cX_k = \cY$ for
  some $\cY \in \bM$.  Suppose further that $(\cX_n')_{n \in \bN}$ is
  a sequence that is obtained by re-ordering the sequence $(\cX_n)_{n
    \in \bN}$.  Then, $\lim_{n \to \infty} \bigboxplus_{k=0}^n \cX_k' = \cY$ also.
% was \label{P:sum_convergence} b)
\item[f)] 
The sequence $(\bigboxplus_{k=0}^n \cX_k)_{n \in \bN}$
  converges to a bounded metric measure space if and only if 
  $\sum_{n \in \bN} \diam(\cX_n) < \infty$.
% was \label{P:sum_convergence} b)
\item[g)] 
The sequence $(\bigboxplus_{k=0}^n \cX_k)_{n \in \bN}$ converges in
$\bM$ 
if 
$\sum_{n \in \bN} \dgpr(\cX_n,\cE)<\infty$ and only if
  $\dgpr(\cX_n,\cE)\to 0$ as $n\to\infty$. 
\end{itemize}
\end{proposition}

\begin{proof}
(a) By the completeness of $(\bM,\dgpr)$,
the convergence of $\bigboxplus_{k=0}^n \cX_k$
as $n \to \infty$ is equivalent to 
\begin{displaymath}
  \lim_{m,n\to\infty} \dgpr\left(\bigboxplus_{k=0}^m \cX_k,
  \bigboxplus_{k=0}^n \cX_k\right) = 0.
\end{displaymath}
However, if $m < n$, then Lemma~\ref{lemma:gpr2} gives
\[
\dgpr\left(\bigboxplus_{k=0}^m \cX_k,
  \bigboxplus_{k=0}^n \cX_k\right)
=
\dgpr\left(\bigboxplus_{k=m+1}^n \cX_k, \cE\right).
\]

\noindent
(b) It suffices to prove the claim for $D$ because the
claim for $\Deltar$ will then follow from
Lemma~\ref{L:chi1_chiA}(c).

Suppose that $\sum_{k \in \bN} D(\cX_k)<\infty$.
For $m < n$, 
\[
\begin{split}
\dgpr\left(\bigboxplus_{k=m+1}^n \cX_k,\cE\right) 
& \le \sqrt{C D\left(\bigboxplus_{k=m+1}^n \cX_k\right)} \\
& = \sqrt{C \left(\sum_{k=m+1}^n D(\cX_k)\right)} \\
\end{split}
\]
for some constant $C$
by parts 
(b) and (c) of Lemma~\ref{L:chi1_chiA}.  It is then 
a consequence of part (a) that
$\bigboxplus_{k=0}^n \cX_k$ converges as $n \to \infty$.

Conversely, if $\lim_{n \to \infty} \bigboxplus_{k=0}^n \cX_k = \cY$
exists, then 
$\sum_{k=0}^n D(\cX_k)
= 
D(\bigboxplus_{k=0}^n \cX_k)
\rightarrow
D(\cY)$
by Lemma~\ref{L:Gromov_weak}.

\noindent
(c) Suppose that $\bigboxplus_{k=0}^n \cX_k \le \cZ$ for
all $n \in \bN$.  It follows from
Lemma~\ref{L:semicharacters_inequality}(b) that 
$\sum_{k=0}^n D(X_k) = D(\bigboxplus_{k=0}^n \cX_k) \le D(\cZ)$ for all
$n \in \bN$, and so part (b) gives that
$\bigboxplus_{k=0}^n \cX_k$ converges as $n \to \infty$.  We
note that an alternative proof of this direction can be given along
the lines of the proof of part (d).

Conversely, suppose that $\lim_{n \to \infty} \bigboxplus_{k=0}^n \cX_k =: \cY$ exists.  We know from one direction of part (b) that 
$\sum_{k \in \bN} D(\cX_k)<\infty$ so that 
$\sum_{k=m+1} D(\cX_k)<\infty$ and hence, 
by the other direction of part (b), 
$\lim_{n \to \infty} \bigboxplus_{k=m+1}^n \cX_k =: \cY_m$ exists
for all $m \in \bN$.  We have 
$\bigboxplus_{k=0}^m \cX_k \boxplus \cY_m = \cY$ for all $m \in \bN$
and hence 
$\bigboxplus_{k=0}^m \cX_k \le \cY$ for all $m \in \bN$.
We note that Proposition~\ref{P:cancellation}(b) can be used
to give an alternative proof of this direction.

\noindent
(d) By Lemma~\ref{L:below_Z_compact}(a)  any subsequence of
$(\cY_n)_{n \in \bN}$ has a further subsequence that converges.  For any
$A \in \bA$, the sequence $(\chi_A(\cY_n))_{n \in \bN}$ is
nondecreasing by Lemma~\ref{L:semicharacters_inequality}(b)
and hence convergent.  By Lemma~\ref{L:Gromov_weak}, all of the
convergent subsequences produced in this manner
 converge to the same limit, and so
the sequence $(\cY_n)_{n \in \bN}$ itself converges to that limit.
  
\noindent
(e)  It follows from Lemma~\ref{L:Gromov_weak} that 
$\sum_{n \in \bN} D_A(\cX_n) = D_A(\cY)$.
It is well-known that all rearrangements of a convergent
sequence with nonnegative terms converge to the same limit.
Thus, 
$\sum_{n \in \bN} D_A(\cX_n') 
= \sum_{n \in \bN} D_A(\cX_n) 
= D_A(\cY)$,
implying that 
$\lim_{n \to \infty} \chi_A(\bigboxplus_{k=0}^n \cX_k') 
= \chi_A(\cY)$ and hence, by Lemma~\ref{L:Gromov_weak}, that
$\lim_{n \to \infty} \bigboxplus_{k=0}^n \cX_k' = \cY$.

\noindent
(f) Suppose that 
$\lim_{n \to \infty} \bigboxplus_{k=0}^n \cX_k = \cY$,
where $\cY$ is bounded.
Since  $\bigboxplus_{k=0}^n \cX_k \le \cY$,
$\sum_{k=0}^n \diam(\cX_k)  = \diam(\bigboxplus_{k=0}^n \cX_k) \le \diam(\cY)$,  and so $\sum_{n \in \bN} \diam(\cX_n) < \infty$. 

Conversely, suppose that $\sum_{n \in \bN} \diam(\cX_n) < \infty$.
It follows from Lemma~\ref{L:chi1_chiA}(d) that 
$\sum_{n \in \bN} D(\cX_n) < \infty$ and hence
$\bigboxplus_{k=0}^n \cX_k$ converges to $\cY\in\bM$ as $n \to \infty$.
 
%Conversely, suppose that $\sum_{n \in \bN} \diam(\cX_n) < \infty$. For $m < n$
%we have from Lemma~\ref{lemma:gpr1} and part (b) of Lemma~\ref{lemma:diam} that
%\[
%\begin{split}
%&\dgpr(\bigboxplus_{k=0}^m \cX_k, \bigboxplus_{k=0}^n \cX_k) \\
%& \quad \le 
%\dgpr(\cE, \cX_{m+1} \boxplus \cdots \boxplus \cX_n) \\
%& \quad \le 
%\diam(\cX_{m+1} \boxplus \cdots \boxplus \cX_n) \\
%& \quad = \diam(\cX_{m+1}) + \cdots + \diam(\cX_n). \\
%\end{split}
%\]
%It follows that the partial sums of $(\cX_n)_{n \in \bN}$ form a
%Cauchy sequence and so, by the completeness of $(\bM, \dgpr)$, $\cX_0
%\boxplus \cdots \boxplus \cX_n$ converges to $\cY\in\bM$ as $n \to \infty$.
The diameter is lower
semicontinuous by Lemma~\ref{L:Diam-lsc} and so $\diam(\cY)\leq \liminf
\sum_{k=0}^n \diam(\cX_k)<\infty$. 

%(c) The sequence of partial sums converges if and only of it is
%fundamental, so that, by Lemma~\ref{lemma:gpr2},
%\begin{displaymath}
%  \dgpr(\bigboxplus_{k=0}^n \cX_k,
%  \cX_0\boxplus\cdots\boxplus\cX_{n+k})
%  =\dgpr(\cX_n\boxplus\cdots\boxplus\cX_{n+k},\cE)\leq \epsilon
%\end{displaymath}
%for all $k\in\bN$ and sufficiently large $n$. By
%Lemma~\ref{L:GPr-inequality}, 
%\begin{displaymath}
%  \dgpr(\cX_n,\cE)\leq \dgpr(\cX_n\boxplus\cdots\boxplus\cX_{n+k})
%  \leq \epsilon.
%\end{displaymath}
%Finally, the triangle inequality together with Lemma~\ref{lemma:gpr2}
%yield that
%\begin{displaymath}
%  \dgpr(\cX_n\boxplus\cdots\boxplus\cX_{n+k},\cE)
%  \leq \sum_{m=n}^{n+k} \dgpr(\cX_n,\cE).
%\end{displaymath}

\noindent
(g)  This part is immediate from part (a) and the observation that
$\dgpr(X_n,\cE) 
\le \dgpr(\bigboxplus_{k=m+1}^n, \cE)
\le \sum_{k=m+1}^n \dgpr(X_k,\cE)$
by Lemma \ref{L:GPr-inequality} and Lemma~\ref{lemma:gpr1}.
\end{proof}

%The sequence $\cX^{\boxplus n}$ does  not converge.

\begin{remark}
  Proposition~\ref{P:convergence_facts}(e) gives that if
  $(\cX_s)_{s \in S}$ is a countable collection of elements of $\bM$,
  then the existence of $\lim_{n \to \infty} \bigboxplus_{k=0}^n \cX_{s_k}$ 
  for some listing $(s_n)_{n \in \bN}$ implies the
  existence for any other listing, with the same value for the limit.
  We will therefore unambiguously denote the limit when it exists by
  the notation $\bigboxplus_{s \in S} \cX_s$.  Moreover, a necessary
  and sufficient condition for $\bigboxplus_{s \in S} \cX_s$ to exist
  is that $\sum_{s \in S} D(\cX_s) < \infty$.
\end{remark}

% It is possible to formulate the condition for relative compactness of
% a family of metric measure spaces using the semicharacters. 

% \begin{proposition}
%   A family $\{\cX_i\}\subset \bM$ is relatively compact if and only if 
%   \begin{displaymath}
%     \inf_{i} \chi_1(\cX_i)>0\,.
%   \end{displaymath}
% \end{proposition}
% \begin{proof}
  
% \end{proof}

We finish this section with a technical result 
that will be used to handle certain measurability issues in
Section~\ref{S:factorization_measure}.
We use the notation $\cV^{\boxplus n}$
for $\cV \in \bM$ and $n \in \bN$ to denote 
$\cV \boxplus \cdots \boxplus \cV$, where there are $n$ terms and
we adopt the convention that this quantity is $\cE$ for $n=0$.

\begin{corollary}
\label{C:partial_order_closed}
\begin{itemize}
\item[a)]
For all $n \in \bN$,
the set $\{(\cX,\cY) \in \bM^2 : \cY^{\boxplus n} \le \cX\}$ is closed.
\item[b)]
The function $M : \bM^2 \to \bN$ defined by 
$M(\cX,\cY) = \max\{n \in \bN : \cY^{\boxplus n} \le \cX\}$
is upper semicontinuous and hence Borel.
\end{itemize}
\end{corollary}

\begin{proof}
  Part (a) is immediate from Proposition~\ref{P:cancellation}(b) for
  $n=1$. If $n\geq2$, let $\cY_k^{\boxplus n}\boxplus \cW_k=\cX_k$ for
  all $k$. If $\cX_k\to\cX$ and $\cY_k\to\cY$, then $\cY_k^{\boxplus
    n}\to \cY^{\boxplus n}$ and the statement again follows from
  Proposition~\ref{P:cancellation}(b).

  For part (b), $\{(\cX, \cY) \in \bM^2 : M(\cX, \cY) \ge n\} =
  \{(\cX,\cY) \in \bM^2 : \cY^{\boxplus n} \le \cX\}$ is a closed set
  for all $n \in \bN$ by part (a), and this is equivalent to the upper
  semicontinuity of $M$.
\end{proof}

\section{Irreducibility and infinite divisibility}
\label{sec:irred-infin-divis}

\begin{definition}
  An element $\cX \in \bM$ is {\em irreducible} if $\cX \ne \cE$ and
  $\cY \le \cX$ for $\cY \in \bM$ implies that $\cY$ is either $\cE$
  or $\cX$ (see \cite[Section 1]{MR1503427}). We write $\bI$ for the
  set of irreducible elements of $\bM$.
\end{definition}

It is not clear {\em a priori} that $\bI$ is nonempty.  For example,
the semigroup $\bR_+$ with the usual addition operation has no
irreducible elements in the sense of the general definition in
\cite{MR1503427}.  The following two results show that $\bI$ is
certainly nonempty.

\begin{proposition}
  \label{P:irreducible_dense}
  The sets $\bI$ and $\bM \setminus \bI$ are dense subsets of $\bM$.
  Moreover, the set $\bI$ is a $G_\delta$
%%% Borel 
  subset of $\bM$.
\end{proposition}

\begin{proof}
It is easy to see that $\bM \setminus \bI$ is a dense subset of $\bM$:
for any $\cX \in \bM$ and $\cZ \in \bM \setminus \{\cE\}$ the elements
$\cX_n := \cX \boxplus (\frac{1}{n} \cZ)$ belong to $\bM \setminus \bI$
and converge to $\cX$ as $n \to \infty$.

  We next show that $\bI$ is dense in $\bM$. As in the proof of
  \cite[Proposition 5.6]{grev:pfaf:win09}, the subset of $\bF
  \subset \bM$ consisting of compact metric measure spaces with
  finitely many points is dense in $\bM$.  If we are given a finite
  metric measure space $(W,r_W,\mu_W)$, then convergence of a sequence
  of probability measures in the Prohorov metric on $(W,r_W)$ is just
  pointwise convergence of the probabilities assigned to each point of
  $W$.  The set of probability measures that assign positive
  probability to all points of $W$ is thus just the relative interior
  of the $(\# W - 1)$-dimensional simplex thought of as a subset of
  $\bR^{\# W}$ equipped with the usual Euclidean topology.  Suppose
  that $(W, r_W)$ is isometric to $(U \times V, r_U \oplus r_V)$ for
  some nontrivial finite compact metric spaces $(U, r_U)$ and $(V,
  r_V)$ -- if this is not the case, then $(W,r_W,\mu_w)$ is already
  irreducible.  The probability measures on $U \times V$ that are of
  the form $\mu_U \otimes \mu_V$ form a $(\# U - 1) + (\# V -
  1)$-dimensional surface in the $(\# U \times \# V - 1)$-dimensional
  simplex of probability measures on $U \times V$ and, in particular,
  the former set is nowhere dense.  Thus, even if $(W, r_W)$ is
  isometric to $(U \times V, r_U \oplus r_V)$, any probability
  measure on $W$ that is the isometric image of a probability measure
  on $U \times V$ of the form $\mu_U \otimes \mu_V$ is arbitrarily
  close to probability measures on $W$ that are not isometric images
  of probability measures of this form, and it follows that $\bI$ is
  dense in $\bM$.

  We now show that the set $\bI$ is a $G_\delta$.  This is equivalent
  to showing that $\bM \setminus \bI$ is an $F_\sigma$.

  Let $\chi_1$ be the semicharacter defined by \eqref{eq:chi-1}.
  Recall that $\chi_1(\cX)=1$ if and only if $\cX=\cE$.  For $0 <
  \epsilon < \frac{1}{2}$ set
  \[
  \bL_\epsilon := 
  \{
  \cX \in \bM : \exists \cY \le \cX, \,
  \chi_1(\cX)^{1-\epsilon} \le \chi_1(\cY) \le \chi_1(\cX)^\epsilon
  \}.
  \]
  Note that $\bL_{\epsilon'} \supseteq \bL_{\epsilon''}$ for
  $\epsilon' \le \epsilon''$ and $\bigcup_{0 < \epsilon < \frac{1}{2}}
  \bL_\epsilon = \bM \setminus \bI$, so it suffices to show that the
  $\bL_\epsilon$ are closed.  Suppose that $(\cX_n)_{n \in \bN}$ is a
  sequence of elements of $\bL_\epsilon$ that converges to $\cX \in
  \bM$.  For each $n \in \bN$ there exist $\cY_n$ and $\cZ_n$ in $\bM$
  such that $\cX_n = \cY_n \boxplus \cZ_n$ and
  $\chi_1(\cX_n)^{1-\epsilon} \le \chi_1(\cY_n) \le \chi_1(\cX_n)^\epsilon$.
  By Lemma~\ref{L:below_Z_compact}(a) and 
  Proposition~\ref{P:cancellation}(b), there is a subsequence $(n_k)_{k
    \in \bN}$ such that $\lim_{k \to \infty} \cY_{n_k} = \cY$ and
  $\lim_{k \to \infty} \cZ_{n_k} = \cZ$ for $\cY, \cZ \in \bM$ such
  that $\cX = \cY \boxplus \cZ$.  Thus, $\cY \le \cX$ and
  $\chi_1(\cX)^{1-\epsilon} \le \chi_1(\cY) \le \chi_1(\cX)^\epsilon$, so
  that $\cX \in \bL_\epsilon$, as required.
%We now show that $\bI$ is Borel. Consider the set
%\[
%\bL := \{(\cW, \cY, \cZ) \in (\bM \setminus \{\cE\})^2 \times \bM 
%: \cW \boxplus \cY = \cZ\}.
%\]
%If we define $\pi: \bM^3 \to \bM$ to be the
%projection onto the third coordinate, then 
%$\bM \setminus \bI = \pi(\bL)$, and so it
%suffices to show that the set $\pi(\bL)$ is Borel.
%For a fixed $\cZ \in \bM$,
%\[
%\begin{split}
%& \{(\cW,\cY) \in \bM^2 : (\cW, \cY, \cZ) \in \bL\} \\
%& \quad =
%\{(\cW,\cY) \in \bM^2 : \cW \boxplus \cY = \cZ\} \\
%& \qquad \cap \bigcup_{n \in \bN} 
%(\bM \setminus \{\cX \in \bM : \dgpr(\cE,\cX) < 2^{-n}\})^2. \\
%\end{split}
%\]
%The set $\{(\cW,\cY) \in \bM^2 : \cW \boxplus \cY = \cZ\}$ is compact
%by essentially the argument in the proof of Lemma~\ref{L:below_Z_compact},
%and hence the set $\{(\cW,\cY) \in \bM^2 : (\cW, \cY, \cZ) \in \bL\}$
%is a countable  union of compact sets.  It follows from a
%a result on projections of Borel sets with
%$\sigma$-compact sections (see 
%\cite[Theorem 18.18 \& Theorem 35.46]{MR1321597} or
%\cite[Theorem 5.12.1]{MR1619545}) which
%generalize results of of Arsenin and Kunugui for
%subsets of the plane that $\pi(\bL)$
%is Borel, as required.
\end{proof}

A theorem of Alexandrov, see \cite[Theorem~3.11]{MR1321597}, says that
a subspace of a Polish space is Polish in the relative topology if and
only if it is a $G_\delta$-set; therefore, the space $\bI$ with
the relative topology inherited from $\bM$ is Polish.  

\begin{remark}
It is not difficult to construct concrete examples of irreducible
elements of $\bM$.  

We first recall that a metric space $(W, r_W)$ is {\em totally geodesic} 
if for any pair of points $w', w'' \in W$ there is a 
unique map $\phi : [0,r_W(w',w'')] \to W$ such that $\phi(0) = w'$,
$\phi(r_W(w',w'')) = w''$ and $r_W(\phi(s),\phi(t)) = |s-t|$
for $s,t \in [0,r_W(w',w'')]$; that is, any two points of $W$
are joined by a unique geodesic segment.  

Any nontrivial closed subset $X$
of a totally geodesic, complete, separable metric space $W$ is 
irreducible no matter what measure it
is equipped with because such a space $(X,r_W)$ cannot
be isometric to a space of the form $(Y \times Z, r_Y \oplus r_Z)$
for nontrivial $Y$ and $Z$.
To see this, suppose that the claim is false.  There will then be
four distinct points $a,b,c,d$ in $X$ that are isometric images 
of points of the form $(y',z')$, $(y'',z')$, $(y',z'')$, $(y'',z'')$ in
$Y \times Z$.  Suppose that $(X,r_W)$ is a closed subset of
the totally geodesic, complete, separable metric space $(W, r_W)$.  We have
\[
r_W(a,b) = r_W(c,d),
\]
\[
r_W(a,c) = r_W(b,d),
\]
\[
r_W(a,d) = r_W(a,b) + r_W(b,d),
\]
\[
r_W(a,d) = r_W(a,c) + r_W(c,d),
\]
\[
r_W(b,c) = r_W(a,b) + r_W(c,a),
\]
and
\[
r_W(b,c) = r_W(b,d) + r_W(c,d).
\]
It follows from the third and fourth equalities that $b$ and $c$
are on the geodesic segment between $a$ and $d$.  We may therefore suppose
that $(W, r_W)$ is a closed subinterval of $\bR$ 
and, without loss of generality, that 
$a < b < c < d$. The fifth and sixth equalities are then impossible.

There are many totally geodesic, complete, separable  metric spaces.
A Banach space $(X, \|\,\|)$
is totally geodesic if and only if it is {\em strictly convex}; that is, 
$x \ne y$ and $\| x' \| = \| x'' \| = 1$ imply that 
$\| a x' + (1 - a) x'' \| < 1$ for all $0 < a < 1$ 
\cite[Section 3.I.1]{MR889253}.  Strict convexity of $(X, \|\,\|)$
is implied by {\em uniform convexity}; that is, for every $\epsilon > 0$
there exists a $\delta > 0$ such that $\| x' \| = \| x'' \| = 1$
and $\|x'-x''\| \ge \epsilon$ 
imply $\|\frac{x'+x''}{2}\| \le 1 - \delta$.  Any Hilbert space
is uniformly convex and the Banach spaces 
$L^p(S,\cS,\lambda)$, $1 < p < \infty$,
where $\lambda$ is a $\sigma$-finite measure, are uniformly convex
\cite[Section 3.II.1]{MR889253}.  Also, any real tree is, by definition,
totally geodesic and any  ultrametric space
is isometric to a subset of a real tree.
\end{remark}

\begin{definition}
  An element of a semigroup is said to be \emph{infinitely divisible}
  if, for each, $n\geq2$, it can be represented as the sum of $n$
  identical summands.
\end{definition}

\begin{proposition}
  \label{P:no_Polish_infinitely_divisible}
  There are no nontrivial infinitely divisible metric measure spaces.
\end{proposition}

\begin{proof}
Suppose that $\cX = (X,r_X,\mu_X)$ 
is a nontrivial infinitely divisible metric measure space.
Thus, for every $n \in \bN$ we have $\cX = \cX_n^{\boxplus 2^n}$
for some metric measure space $\cX_n = (X_n, r_{X_n}, \mu_{X_n})$. 
We may suppose that $X_0 = X$, $r_{X_0} = r_X$
and $\mu_{X_0} = \mu_X$,
and that for all $n \in \bN$ there is an isometry
$\phi_{n,n+1}$ from $X_n$ equipped with $r_{X_n}$ to $X_{n+1}$
equipped with $r_{X_{n+1}} \oplus r_{X_{n+1}}$ such that
the push-forward of $\mu_{X_n}$ by $\phi_{n,n+1}$
is $\mu_{X_{n+1}} \otimes \mu_{X_{n+1}}$.  
Let $\xi_{i}$, $i \in \bN$, be independent identically distributed random
elements of $X$ with common distribution $\mu_X$.
Define $(\xi_{ni1}, \ldots, \xi_{ni2^n})$, $n \in \bN$, $i \in \bN$,
 recursively by
$\xi_{0i1} = \xi_i$ and 
$(\xi_{n+1,i,2k-1}, \xi_{n+1,i,2k}) = \phi_{n,n+1}(\xi_{nik})$ for 
$k \in \{1, \ldots,2^n\}$.
The
$\xi_{nik}$, $i \in \bN$, $k \in \{1, \ldots,2^n\}$
are random elements of $X_n$ with distribution $\mu_{X_n}$,
$r_{X_n}(\xi_{nik}, \xi_{njk}) 
= 
r_{X_{n+1}}(\xi_{n+1,i,2k-1}, \xi_{n+1,j,2k-1})
+
r_{X_{n+1}}(\xi_{n+1,i,2k},   \xi_{n+1,j,2k})$,
and consequently
$r_X(\xi_i,\xi_j) = \sum_{k=1}^{2^n} r_{X_n}(\xi_{nik}, \xi_{njk})$.

For $i \ne j$ the nonnegative random variable $r_X(\xi_i,\xi_j)$
is clearly infinitely divisible.  These random variables are not almost surely
zero and they are identically distributed.  Their common
distribution does not have a nontrivial deterministic component because
that would mean that for some $c>0$ we would have $r_X(\xi_i,\xi_j) \ge c$
for all $i \ne j$, which is impossible because
almost surely for all 
$i \in \bN$ we must have
$\inf_{j \in \bN, \, j \ne i} r_X(\xi_i, \xi_j) = 0$
if $(\xi_h)_{h \in \bN}$
is an independent identically distributed sequence of random
elements of $X$ with common distribution $\mu_X$.
In particular, these random variables are not bounded, because a bounded 
infinitely divisible random variable is almost surely constant.
It follows that the metric $r_X$ is unbounded.

Let $\nu$ be the L\'evy measure associated with the common infinitely
divisible distribution of $r_X(\xi_i,\xi_j)$ for $i \ne j$.  This is a
(nontrivial) measure on $\bR_{++}:=(0,\infty)$ that satisfies
$\int_{\bR_{++}} (x \wedge 1) \, \nu(dx) < \infty$ and it is the limit
as $n \to \infty$ of the measures
\[
\sum_{k=1}^{2^n} \bP\{r_{X_n}(\xi_{nik},\xi_{njk}) \in \cdot\}
=
2^n \int_{X_n^2} \ind\{r_{X_n}(y,z) \in \cdot\} \, \mu_{X_n}^{\otimes 2}(dy,dz),
\]
where the limit is in the sense of vague convergence of measures on $\bR_{++}$.

For $K>0$, set 
\begin{displaymath}
  R_n^K(i,j) := \sum_{k=1}^{2^n} (r_{X_n}(\xi_{nik},\xi_{njk}) \wedge K).
\end{displaymath}
As $n \to \infty$, $R_n^K(i,j)$ converges almost surely to an
infinitely divisible random variable $R^K(i,j)$ with
\[
\bE[R^K(i,j)] = \int_{\bR_{++}} (x \wedge K) \, \nu(dx) < \infty,
\]
and $R^K(i,j) = r_X(\xi_i,\xi_j)$ for all $K$ sufficiently large
almost surely.  The random matrix $(r_X(\xi_i,\xi_j))_{i,j \in \bN}$
satisfies the necessary and sufficient condition \eqref{Eq:Vershik} to
be the matrix of pairwise distances for a sample from a
metric measure space, and it follows easily that the same is true of the
random matrix $(R^K(i,j))_{i,j \in \bN}$.  Because 
the random matrix 
$(R^K(i,j))_{i,j \in \bN}$ is infinitely divisible, 
the underlying metric measure space that gives rise
to this matrix of pairwise distances is also infinitely
divisible.  We may therefore suppose without loss of generality that
the random variables $r_X(\xi_i,\xi_j)$ are integrable.

It is clear from Fubini's theorem that 
\[
\bE[r_X(y, \xi_j)] = \int_X r_X(y,z) \, \mu_X(dz) < \infty,
\quad \text{$\mu_X$-a.e. $y \in X$}.
\]
Because $r_X$ is unbounded, the function
$y \mapsto \int_X r_X(y,z) \, \mu_X(dz)$ 
is also unbounded, and since $\mu_X$ has full support, each of
the random variables $\bE[r_X(\xi_i,\xi_j) \, | \, \xi_i]$, $i \ne j$,
are unbounded.  These random variables are equal for a fixed $i$ as
$j$ varies and as $i$ varies the common values are independent
and identically distributed.  Moreover, 
\[
\bE[r_X(\xi_i,\xi_j) \, | \, \xi_i]
=
\sum_{k=1}^{2^n} \int_{X_n} r_{X_n}(\xi_{nik},z) \, \mu_{X_n}(dz)
\]
for all $n \in \bN$, and so $\bE[r_X(\xi_i,\xi_j) \, | \, \xi_i]$
is infinitely divisible
and, being unbounded, this
random variable cannot be constant almost surely.

Given $\epsilon > 0$, set 
\[
I_{nik}^\epsilon 
= 
\ind\left\{\int_{X_n} r_{X_n}(\xi_{nik},z) \, \mu_{X_n}(dz) > \epsilon\right\}.
\]
For $\epsilon$ sufficiently small, 
$\sum_{k=1}^{2^n} I_{nik}^\epsilon$ converges almost surely
as $n \to \infty$ to a nontrivial random variable 
$J_i^\epsilon$ that has
a Poisson distribution.  Moreover, for $\epsilon', \epsilon'' > 0$ and
$i \ne j$, $\sum_{k=1}^{2^n} I_{nik}^{\epsilon'} I_{njk}^{\epsilon''} = 0$
for all $n$ sufficiently large almost surely by the independence of 
$\{\xi_{nik} : n \in \bN, \, 1 \le k \le 2^n\}$ and 
$\{\xi_{njk} : n \in \bN, \, 1 \le k \le 2^n\}$.  

By the triangle inequality, 
\[
\begin{split}
\int_{X_n} r_{X_n}(y'',z) \, \mu_{X_n}(dz)
& \ge
\int_{X_n} \left[r_{X_n}(y',z) - r_{X_n}(y',y'')\right] \, \mu_{X_n}(dz) \\
% & =
% \int_{X_n} r_{X_n}(y',z) \, \mu_{X_n}(dz) \\
% & \quad - \int_{X_n} r_{X_n}(y',y'') \, \mu_{X_n}(dz) \\
& =
\int_{X_n} r_{X_n}(y',z) \, \mu_{X_n}(dz) - r_{X_n}(y',y'')\\
\end{split}
\]
and hence
\[
r_{X_n}(y',y'')  \ge 
\int_{X_n} r_{X_n}(y',z) \, \mu_{X_n}(dz) 
- \int_{X_n} r_{X_n}(y'',z) \, \mu_{X_n}(dz). 
\]
Therefore, if $\int_{X_n} r_{X_n}(y',z) \, \mu_{X_n}(dz) > \epsilon'$
and $\int_{X_n} r_{X_n}(y'',z) \, \mu_{X_n}(dz) \le \epsilon''$
for $\epsilon' > \epsilon'' > 0$, then 
$r_{X_n}(y',y'') > \epsilon' - \epsilon''$.
Thus,
\[
\begin{split}
r_X(\xi_i, \xi_j)
& =
\sum_{k=1}^{2^n} r_{X_n}(\xi_{nik}, \xi_{njk}) \\
& \ge
\sum_{k=1}^{2^n} I_{nik}^{\epsilon'} (1 - I_{njk}^{\epsilon''}) 
r_{X_n}(\xi_{nik}, \xi_{njk}) \\
& \ge
\sum_{k=1}^{2^n} I_{nik}^{\epsilon'} (1 - I_{njk}^{\epsilon''}) 
(\epsilon' - \epsilon'') \\
\end{split}
\]
and so on the event
$\{\sum_{k=1}^{2^n} I_{nik}^{\epsilon'} I_{njk}^{\epsilon''} = 0\}$
\[
r_X(\xi_i, \xi_j) 
\ge 
(\epsilon' - \epsilon'') \sum_{k=1}^{2^n} I_{nik}^{\epsilon'}.
\]
Consequently,
\[
r_X(\xi_i, \xi_j) 
\ge 
\epsilon J_i^\epsilon
\]
for all $i \ne j$ almost surely. This, however, is impossible because
if $(\xi_h)_{h \in \bN}$ is an independent identically distributed
sequence of random elements of $X$ with common distribution $\mu_X$,
then almost surely for all $i \in \bN$ we must have $\inf_{j \in \bN,
  \, j \ne i} r_X(\xi_i, \xi_j) = 0$.
\end{proof}

\begin{remark}
  % Theorem~\ref{T:exist_unique_factorization} implies that there are no
  % non-trivial infinitely divisible metric measure spaces. Indeed, a
  % decomposition of the form $\cX = \cX_n \boxplus \cdots \boxplus
  % \cX_n$ is only possible if $n$ divides each of the multiplicities
  % with which the various irreducible elements appear in the
  % factorization of $\cX$. 
  In the case of bounded metric measure spaces, a simpler and more direct
  proof of Proposition~\ref{P:no_Polish_infinitely_divisible} is to
  note that if $\cX=\cX_n^{\boxplus n}$ for all $n$, then the
  push-forward of the probability measure $\mu_X^{\otimes 2}$ by the
  map $(x',x'') \to r_X(x',x'')$ is an infinitely divisible
  probability measure supported on $[0, \diam(\cX)]$ and hence it must
  be a point mass at zero because any
  infinitely divisible probability measure with bounded support is a
  point mass and if that point mass was not at zero, then the distribution
  of $(r_X(\xi_i,\xi_j))_{i,j \in \bN}$ for an i.i.d. sequence
  $(\xi_k)_{k \in \bN}$ with common distribution $\mu_X$
  would certainly not satisfy the condition \eqref{Eq:Vershik}.
  %, a fact that, as we noted in
  %Remark~\ref{R:bounded_sums_negligible}, has a simple direct proof.
  %An argument along these lines can also be used to establish
  %Corollary~\ref{C:no_cts_increasing}.
\end{remark}

\section{Arithmetic properties}
\label{S:arithmetic}

The theory of \emph{Delphic semigroups} was developed in
\cite{MR0229746,MR0263128} to generalize the decomposability
properties of probability distributions with respect
to convolution to an abstract setting. In the
following we show that $(\bM,\boxplus)$ is a Delphic semigroup. Let us
associate with each converging sequence $(\cX_n)_{n\in\bN}$ in $\bM$
its limit $L((\cX_n))$. If $\cY_n\leq\cX_n$ for all $n$, then
$(\cY_n)_{n\in\bN}$ is a subset of
$\bigcup_{\cZ\in\bS}\{\cY\in\bM:\cY\leq \cZ\}$ for the compact set
$\bS:=(\cX_n)_{n\in\bN}$ and so it is compact by
Lemma~\ref{L:below_Z_compact}(a). Thus, $(\cY_n)_{n\in\bN}$ admits a
convergent subsequence, so that the condition $(A')$ from
\cite{MR0263128} holds.

For each $A\in\bA$, the function $D_A=-\log \chi_A$ is a continuous
homomorphism from $(\bM,\boxplus)$ to $(\bR_+,+)$. In particular, the
function $D$ from \eqref{Eq:D_1} has this property.  By
Lemma~\ref{L:chi1_chiA}(a), for each $A\in\bA$ and each $\epsilon>0$
there is $\delta>0$ such that, for any $\cX\in\bM$ satisfying
$D(\cX)\leq\delta$ one has $D_A(\cX)\leq\epsilon$.
% This is easy to see for $A:=(a)$ being a
% $1\times1$ matrix: the statement is evident for $a<1$ and for $a>1$
% follows from the Jensen inequality. For a general matrix $A\in\bA$
% with $n$ non-zero rows, the triangle inequality yields that 
% \begin{displaymath}
%   \sum_{1\leq i<j\leq n} a_{ij}r_X(x_i,x_j)
%   \leq a\sum_{i=2}^n r_X(x_1,x_i)
% \end{displaymath}
% for a certain constant $a$. Therefore,
% \begin{align*}
%   \chi_A(\cX)&\geq \int_X \left( \int_X e^{-ar_X(x,y)}\mu_X(dy)\right)^{n-1} \mu_X(dx)\\
%   &\geq \left( \int_{X^2} e^{-ar_X(x,y)}\mu_X(dx)\mu_X(dy)\right)^{n-1}\\
%   &=(\chi_{(a)}(\cX))^{n-1}\,.
% \end{align*}
% % If $\chi_{(a)}(\cX)\geq 1-\delta$, then
% % \begin{displaymath}
% %   \mu_X\{x\in X: \int_X e^{-ar_X(x,y)}\mu_X(dy)\geq 1-\delta\}\geq 1-\delta\,,
% % \end{displaymath}
% % whence $\chi_A(\cX)\geq 1-\epsilon$ if $\delta$ is chosen sufficiently
% % small. 
% Thus, for each $A\in\bA$ and $\epsilon>0$, it is possible to
% find $\delta>0$ such that $D(\cX)\leq\delta$ implies $D_A(\cX)\leq\epsilon$.
If $(A_k)_{k\in\bN}$ is a countable subset of $\bA$, such that 
$(A_k)_{k\in\bN} \cap \bR_+^{\binom{n}{2}}$ is dense in $\bR_+^{\binom{n}{2}}$ for all $n\geq2$,
then the values $D_A(\cX)$ uniquely determine $\cX$, so that the
homomorphisms $D_k$ satisfy the condition $(H)$ of \cite{MR0263128}.
By \cite[Theorem~3]{MR0263128}, the semigroup $(\bM,\boxplus)$ is
sequentially Delphic; in particular, it satisfies the (CLT) condition
that requires that the limit of any converging null-array is
infinitely divisible. By \cite[Theorem~II]{MR0229746}, each element
$\cX$ of $(\bM,\boxplus)$ is either irreducible or has an irreducible
factor or is infinitely divisible. The last is impossible by
Proposition~\ref{P:no_Polish_infinitely_divisible}, so the
next result holds.

\begin{proposition}
\label{P:irreducible_exist}
Given any $\cX \in \bM \setminus \{\cE\}$, 
there exists $\cY \in \bI$
with $\cY \le \cX$.
\end{proposition}

The prime numbers are the analogue of irreducible elements for the
semigroup of positive integers equipped with the usual multiplication.
The key to proving the Fundamental Theorem of Arithmetic (that every
positive integer other than $1$ has a factorization into primes that
is unique up to the order of the factors) is a lemma due to Euclid
which says that if a prime number  divides the product of two
positive integers, then it must divide one of the factors.  For
general commutative semigroups, the term ``prime'' is usually reserved
for elements that exhibit the generalization of this property (see,
for example, \cite{MR1503427}).  Accordingly, we say that an element
$\cX \in \bM \setminus \{\cE\}$ is {\em prime} if $\cX \le \cY
\boxplus \cZ$ for $\cY, \cZ \in \bM$ implies that $\cX \le \cY$ or
$\cX \le \cZ$.  Prime elements are clearly irreducible, but the
converse is not {\em a priori} true and there are commutative,
cancellative semigroups where the analogue of the converse is false.

Before showing that the notions of irreducibility and primality
coincide in our setting, we need the following elementary lemma which
we prove for the sake of completeness.

\begin{lemma}
\label{L:independence_breakup}
Let $\xi_{00}, \xi_{01}, \xi_{10}, \xi_{11}$ be random elements of the
respective metric spaces $X_{00}, X_{01}, X_{10}, X_{11}$.  Suppose
that the pairs $(\xi_{00}, \xi_{01})$ and $(\xi_{10}, \xi_{11})$ are
independent and that the pairs $(\xi_{00}, \xi_{10})$ and $(\xi_{01},
\xi_{11})$ are independent.  Then, $\xi_{00}, \xi_{01}, \xi_{10},
\xi_{11}$ are independent.
\end{lemma}

\begin{proof}
Suppose that $f_{ij}: X_{ij} \to \bR$, $i,j \in \{0,1\}$,
are bounded Borel functions.  Using first the independence
of $(\xi_{00}, \xi_{01})$ and $(\xi_{10}, \xi_{11})$, 
and then the independence of
$(\xi_{00}, \xi_{10})$ and $(\xi_{01}, \xi_{11})$,
we have
\[
\begin{split}
& 
\bE[f_{00}(\xi_{00}) f_{01}(\xi_{01}) f_{10}(\xi_{10}) f_{11}(\xi_{11})] \\
& \quad =
\bE[f_{00}(\xi_{00}) f_{01}(\xi_{01})] \bE[f_{10}(\xi_{10}) f_{11}(\xi_{11})] \\
& \quad =
\bE[f_{00}(\xi_{00})] \bE[f_{01}(\xi_{01})] 
\bE[f_{10}(\xi_{10})] \bE[f_{11}(\xi_{11})], \\
\end{split}
\]
as required.
\end{proof}

\begin{proposition}
\label{P:indecomp_prime}
All irreducible elements of $\bM$ are prime.
Moreover, if $(\cY_n)_{n \in \bN}$ is a sequence
of elements of $\bM$ such that 
$\lim_{n \to \infty} \bigboxplus_{k=0}^n \cY_k = \cY$
exists and $\cX \in \bI$ is such that $\cX \le \cY$, then
$\cX \le \cY_n$ for some $n \in \bN$.
\end{proposition}

\begin{proof}
Consider the first claim.
Suppose that $\cX \in \bM$ is irreducible and
$\cX \le \cY \boxplus \cZ$ for some $\cY, \cZ \in \bM$.

From Proposition~\ref{P:cancellation}(a) we have 
$\cY \boxplus \cZ = \cW \boxplus \cX$ for some unique $\cW \in \bM$.
From the remarks at the end of \cite{MR1192390}, we may suppose that
there are metric spaces $(Y', r_{Y'})$, $(X', r_{X'})$,
$(X'', r_{X''})$ and $(Z'', r_{Z''})$ such that
$(Y, r_Y) = (Y' \times X', r_{Y'} \oplus r_{X'})$,
$(Z, r_Z) = (X'' \times Z'', r_{X''} \oplus r_{Z''})$,
$(X, r_X) = (X' \times X'', r_{X'} \oplus r_{X''})$
and
$(W, r_W) = (Y' \times Z'', r_{Y'} \oplus r_{Z''})$,
so that
$(Y \times Z, r_Y \oplus r_Z) = (W \times X, r_W \oplus r_X)
= (Y' \times X' \times X'' \times Z'', 
r_{Y'} \oplus r_{X'} \oplus r_{X''} \oplus r_{Z''})$ 
(see also \cite{MR904402} for an analogous result concerning the
existence of a common refinement of two Cartesian factorizations of a (possibly
infinite) graph and \cite{MR1787898} for the case of finite metric spaces).
It follows from Lemma~\ref{L:independence_breakup} that
there are probability measures $\mu_{Y'}$, $\mu_{X'}$,
$\mu_{X''}$ and $\mu_{Z''}$ such that 
$\mu_Y = \mu_{Y'} \otimes \mu_{X'}$, 
$\mu_Z = \mu_{X''} \otimes \mu_{Z''}$,
$\mu_X = \mu_{X'} \otimes \mu_{X''}$,
$\mu_W = \mu_{Y'} \otimes \mu_{Z''}$,
and
$\mu_Y \otimes \mu_Z = \mu_W \otimes \mu_X 
= \mu_{Y'} \otimes \mu_{X'} \otimes \mu_{X''} \otimes \mu_{Z''}$.
Thus,
$\cY = \cY' \boxplus \cX''$,
$\cZ = \cX' \boxplus \cZ''$,
$\cX = \cX' \boxplus \cX''$,
$\cW = \cY' \boxplus \cZ''$,
and
$\cY \boxplus \cZ = \cW \boxplus \cX
= \cY' \boxplus \cX' \boxplus \cX'' \boxplus \cZ''$.
This contradicts the irreducibility of $\cX$
unless $\cX' = \cE$ or $\cX'' = \cE$, in which case
$\cX \le \cZ$ or $\cX \le \cY$, thus establishing
the first claim of the proposition.

Turning to the second claim, let $(\cY_n)_{n \in \bN}$, $\cY \in \bM$
and $\cX \in \bI$ satisfy the hypotheses of the claim.  By
Proposition~\ref{P:cancellation}(b), for each $n \in \bN$ we have 
$\cY = \bigboxplus_{k=0}^n Y_k  \boxplus \cZ_n$ for some unique
$\cZ_n \in \bM$.  If there is no $n \in \bN$ such that $\cX \le
\cY_n$, then, by the first part of the proposition, $\cX \le \cZ_n$
for all $n \in \bN$.  By Proposition~\ref{P:cancellation}(b), this means
that $\cZ_n = \cX \boxplus \cW_n$ for some unique $\cW_n \in \bM$ and
hence $\chi_A(\cZ_n) \le \chi_A(\cX)$ for all $A \in \bA$, see
Lemma~\ref{L:semicharacters_inequality}(b).  However, $\lim_{n \to
  \infty} \chi_A(\bigboxplus_{k=0}^n Y_k) = \chi_A(\cY)$
for all $A \in \bA$ and so $\lim_{n \to \infty} \chi_A(\cZ_n) = 1$ for
all $A \in \bA$, implying that $\chi_A(\cX) = 1$ for all $A \in \bA$.
This, however, is impossible, since it would imply that $\cX = \cE
\notin \bI$.
\end{proof}

The next result is standard, but we include it for the sake of completeness.

\begin{corollary}
\label{C:product_prime_divisors}
Suppose for $\cX \in \cK$ and distinct $\cY_0, \ldots, \cY_n \in \bI$ that
$\cY_k \le \cX$ for $k=0,\ldots,n$.  
Then, $\bigboxplus_{k=0}^n Y_k \le \cX$.
\end{corollary}

\begin{proof}
The proof is by induction.  The statement is certainly true for $n=0$.
Suppose it is true for $n=r$ and consider the case $n=r+1$.  We have 
$\cX = \bigboxplus_{k=0}^r Y_k \boxplus \cW_r$
for some $\cW_r \in \bM$ by the inductive assumption.
Because
$\cY_{r+1} \le \cX = \bigboxplus_{k=0}^r Y_k  \boxplus \cW_r$,
it follows from Proposition~\ref{P:indecomp_prime} that
either $\cY_{r+1} \le \cY_k$ for some $k$ with $1 \le k \le r$
 or $\cY_{r+1} \le \cW_r$.  The former alternative is impossible because $\cY_0, \ldots, \cY_r, \cY_{r+1} \in \bI$ are distinct.  
Thus, $\cY_{r+1} \le \cW_r$ and we have $\cW_r = \cY_{r+1} \boxplus \cW_{r+1}$ for some $\cW_{r+1} \in \bM$.  This implies that
$\cX = 
\bigboxplus_{k=0}^r Y_k \boxplus  \cY_{r+1} \boxplus \cW_{r+1}$
and hence  $\bigboxplus_{k=0}^{r+1} Y_k \le \cX$, 
completing the inductive step.
\end{proof}

\begin{theorem}
\label{T:exist_unique_factorization}
Given any $\cX \in \bM \setminus \{\cE\}$, there is either
a finite sequence $(\cX_n)_{n=0}^N$ 
or an infinite sequence
$(\cX_n)_{n=0}^\infty$
of irreducible elements of $\bM$
such that 
$\cX = \bigboxplus_{k=0}^N \cX_k$ in the first case and
$\cX = \lim_{n \to \infty} \bigboxplus_{k=0}^n \cX_k$
in the second.  The sequence is unique up to the order of its terms.
Each irreducible element appears a finite number of times, so the
representation is specified by the irreducible elements that appear
and their finite multiplicities.
\end{theorem}

\begin{proof}
  As $(\bM,\boxplus)$ is a Delphic semigroup,
  \cite[Theorem~III]{MR0229746} yields that each $\cX\in\bM$ admits a
  representation as the sum of irreducible elements. Note that each
  element of the sum appears only a finite number of times, since
  otherwise the sum would diverge by
  Proposition~\ref{P:convergence_facts}(b).

  We now turn to the uniqueness claim.  This may fail because $\cX$
  has two different representations as a finite sum of irreducible
  elements, one representation as a finite sum and another as a limit
  of finite sums, or two different representations as a limit of
  finite sums.  We deal with the last case.  The other two are similar
  and are left to the reader.  Suppose then that two sequences
  $(\cX_n')_{n \in \bN}$ and $(\cX_n'')_{n \in \bN}$ represent $\cX$.
  An argument similar to one above shows that any particular
  irreducible element appears a finite number of times in each
  sequence.  Suppose that $\cY \in \bI$ appears $M'$ times in
  $(\cX_n')_{n \in \bN}$ and $M''$ times in $(\cX_n'')_{n \in \bN}$
  with $M' \ne M''$.  Assume without loss of generality that $M' >
  M''$. We have $\cY^{\boxplus M'} \boxplus \cZ' = \cX = \cY^{\boxplus
    M''} \boxplus \cZ''$, where $\cZ', \cZ'' \in \bM$ are such that
  $\cY \not \le \cZ'$ and $\cY \not \le \cZ''$.  Using
  Proposition~\ref{P:cancellation}(a), $\cY^{\boxplus (M'-M'')} \boxplus
  \cZ' = \cZ''$.  By Proposition~\ref{P:indecomp_prime}, $\cY$ is
  prime, so that it divides one of the factors in the representation
  of $\cZ''$ meaning that so $\cY\leq \cZ''$, contrary to the
  assumption.
\end{proof}

% \begin{remark}
%   The existence of the prime decomposition in
%   Theorem~\ref{T:exist_unique_factorization} follows from the Delphic
%   theory \cite[Theorem~III]{MR0229746}. However, our proof does not
%   use the well ordering principle. 
% \end{remark}

\begin{remark} 
  It is an easy consequence of Theorem~\ref{T:exist_unique_factorization}
  that, for the partial order $\le$, every pair of elements of $\bM$
  has a {\em join} (that is, a least upper bound) and a {\em meet}
  (that is, a greatest lower bound), and so $\bM$ with these
  operations is a {\em lattice}.  It is not hard to check that this
  lattice is distributive (that is, the meet operation distributes
  over the join operation and vice versa). Furthermore, the
  Gromov--Prohorov distance between $\cX$ and $\cY$ equals the maximum
  of the distances between the meet of $\cX$ and $\cY$ and either
  $\cX$ or $\cY$.
\end{remark}

\begin{remark}
\label{R:factorization_semicharacter}
Given $f: \bI \to [0,1]$, the map $\chi: \bM \to [0,1]$ that sends
$\cX$ to $\prod_n f(\cX_n)$, where $\cX_0, \cX_1, \ldots$
are as in Theorem~\ref{T:exist_unique_factorization}, is
a semicharacter.
\end{remark}

% The next result will be a consequence of the characterization of
% infinitely divisible random elements of $\bM$
% in Theorem~\ref{T:Levy-Hincin-Ito}, but we present it here as
% in illustration of a nonobvious and initially somewhat
% surprising feature of $\bM$.

% \begin{corollary}
%   \label{C:inf-div-deterministic}
%   Suppose that $\cX \in \bM$ is infinitely divisible in the
%   sense that for each positive integer $n$ there exists $\cX_n \in \bM$
%   such that $\cX = \cX_n \boxplus \cdots \boxplus \cX_n$, where the sum has 
%   $n$ terms.  Then, $\cX=\cE$.
% \end{corollary}

The following result will be a key ingredient in the characterization of
the infinitely divisible random elements of $\bM$ 
in Theorem~\ref{T:Levy-Hincin-Ito}.

\begin{corollary}
\label{C:no_cts_increasing}
If $\Phi: \bR_+ \to \bM$ is a continuous function such that 
$\Phi(s) \le \Phi(t)$ for $0 \le s \le t < \infty$, then $\Phi \equiv \cE$.
\end{corollary}

\begin{proof}
  Suppose that $\Phi$ is a function with the stated properties.  If
  $\Phi \not \equiv \cE$, then there exist $0 < u < v < \infty$ such
  that $\Phi(u) < \Phi(v)$.  It follows from
  Theorem~\ref{T:exist_unique_factorization} that there exists $\cY
  \in \bI$ such that the multiplicity of $\cY$ in the factorization of
  $\Phi(v)$ is strictly greater than the multiplicity of $\cY$ in the
  factorization of $\Phi(u)$.  Define $M: \bR_+ \to \bN$ by setting
  $M(s)$, $s \ge 0$, to be the multiplicity of $\cY$ in the
  factorization of $\Phi(s)$.  This function is nondecreasing and so
  there must exist $u \le t \le v$ such that $M(t-) < M(t+)$.  
  Thus, $\Phi(t-\epsilon) \boxplus \cY \boxplus \cdots \boxplus
  \cY \le \Phi(t+\epsilon)$ for all $\epsilon > 0$, where there are
  $M(t+) - M(t-)$ summands in the sum, and this contradicts the
  continuity of $\Phi$ by Lemma~\ref{L:GPr-inequality} and
  Lemma~\ref{lemma:gpr2}.
\end{proof}

The following result is an immediate consequence of the absence of
infinitely divisible metric measure spaces.
%Corollary~\ref{C:inf-div-deterministic}.

\begin{corollary}
  \label{C:no_additive}
  If $\Phi: \bR_+ \to \bM$ is a function such that $\Phi(s) \boxplus
  \Phi(t) = \Phi(s+t)$ for $0 \le s, t < \infty$, then $\Phi \equiv
  \cE$.
\end{corollary}

\begin{remark}
Although Corollary~\ref{C:no_additive} says there are no
nontrivial {\em additive} functions from $\bR_+$ to $\bM$,
there do exist nontrivial {\em superadditive functions}; that is,
functions $\Phi: \bR_+ \to \bM$  such that $\Phi(0) = \cE$
and $\Phi(s) \boxplus \Phi(t) \le \Phi(s+t)$ for $0 \le s, t < \infty$.
For example, take $\cX \in \bM \setminus \{\cE\}$ and set
$\Phi(t) = \cX \boxplus \cdots \boxplus \cX$ for $n \le t < n+1$, $n \in \bN$,
where the sum has $n$ terms and we interpret the empty sum as $\cE$.
We have
\[
\Phi(s) \boxplus \Phi(t) 
= \Phi(\lfloor s \rfloor) \boxplus \Phi(\lfloor t \rfloor)
= \Phi(\lfloor s \rfloor + \lfloor t \rfloor)
\le 
\Phi(s + t).
\]
However, by Corollary~\ref{C:no_cts_increasing} there are no
nontrivial continuous superadditive functions. Furthermore, there are no
superadditive functions $\Phi$ such that $\Phi(t) \ne \cE$ for all 
$t > 0$. 

There are also nontrivial {\em subadditive functions}; that is,
functions $\Phi: \bR_+ \to \bM$  such that $\Phi(0) = \cE$
and $\Phi(s) \boxplus \Phi(t) \ge \Phi(s+t)$ for $0 \le s, t < \infty$.
For example, it suffices to take some $\cX \in \bM \setminus \{\cE\}$ and set
$\Phi(t) = \cX$ for $t > 0$.  However, there are no continuous 
subadditive functions because if $\Phi$ is such a function and $\cY \in \bI$ 
is such that $\cY \le \Phi(t)$, then it follows from 
$\Phi(\frac{t}{2}) \boxplus \Phi(\frac{t}{2}) \ge \Phi(t)$
that $\cY \le \Phi(\frac{t}{2})$ and hence $\cY \le \Phi(\frac{t}{2^n})$
for all $n \in \bN$, but this contradicts the continuity of $\Phi$ at $0$.
\end{remark}

% \query{the following holds in the bounded case.}

% \begin{remark}
% \label{R:group_embedding_redux}
% With Theorem~\ref{T:exist_unique_factorization} in hand, we can give a
% more concrete description of the group $\bG$ described in
% Remark~\ref{R:group_embedding}.  Recall that $\cV^{\boxplus n}$
% denotes the sum of $n$ terms $\cV \boxplus \cdots \boxplus \cV$, and
% we interpret the empty sum as $\cE$.  With this notation, any $\cU \in
% \bM$ has a unique representation as $\cU = \bigboxplus_{\cV \in \bI}
% \cV^{\boxplus n_\cV}$, where $n_\cV = 0$ for all but countably many
% $\cV \in \bI$ and $\sum_{\cV \in \bI} n_\cV \diam(\cV) < \infty$.  Any
% element of $\bG$ corresponds to a unique pair $(\bigboxplus_{\cV \in
%   \bI} \cV^{\boxplus n_\cV^+}, \bigboxplus_{\cV \in \bI} \cV^{\boxplus
%   n_\cV^-})$, where $n_\cV^+ = 0$ and $n_\cV^- = 0$ for all but
% countably many $\cV \in \bI$, $\sum_{\cV \in \bI} n_\cV^+ \diam(\cV) <
% \infty$ and $\sum_{\cV \in \bI} n_\cV^- \diam(\cV) < \infty$, and
% $n_\cV^+ n_\cV^- = 0$ for all $\cV \in \bI$.  We can therefore
% identify an element of $\bG$ with the corresponding object $((n_\cV^+,
% n_\cV^-))_{\cV \in \bI}$.  In terms of this representation, the binary
% operation on $\bG$ transforms the two objects $((m_\cV^+,
% m_\cV^-))_{\cV \in \bI}$ and $((n_\cV^+, n_\cV^-))_{\cV \in \bI}$ into
% the object
% \[
% (
% (m_\cV^+ + n_\cV^+ - [m_\cV^+ + n_\cV^+] \wedge [m_\cV^- + n_\cV^-], 
%  m_\cV^- + n_\cV^- - [m_\cV^+ + n_\cV^+] \wedge [m_\cV^- + n_\cV^-])
% )_{\cV \in \bI}.
% \]
% \end{remark}

\section{Prime factorizations as measures}
\label{S:factorization_measure}

Theorem~\ref{T:exist_unique_factorization} guarantees that any $\cX
\in \bM$ has a unique representation as $\cX = \bigboxplus_k
\cY_k^{\boxplus m_k}$, where the $\cY_k \in \bI$ are distinct, the
integers $m_k$ are positive, and we define the empty sum to be
$\cE$. Since $\bigboxplus_k \cY_k^{\boxplus m_k}$ converges,
$\dgpr(\cY_k,\cE)\to0$ as $k\to \infty$ in case of an infinite
factorization, so that the number of $\cY_k$ outside any neighborhood
of $\cE$ is finite.  It is natural to code such a factorization as the
measure $\Psi(\cX):= \sum_k m_k \delta_{\cY_k}$ on $\bM$ that is
concentrated on $\bI$ and assigns mass $m_k$ to the point $\cY_k$ for
each $k$.

Denote by $\fN$ the family of Borel measures $N$ on $\bM$ such that
$N(\bM \setminus \bI) = 0$ and $N(B) \in \bN$ for every
Borel set $B$ that does not intersect some neighborhood of $\cE$. 
Any $N \in \fN$ can be
represented as the positive integer linear combination of Dirac measures 
\begin{displaymath}
  N =\sum_k m_k \delta_{\cY_k}
\end{displaymath}
for distinct $\cY_k\in\bI$ and positive integers $m_k$, where the sum may be
finite or countably infinite depending on the cardinality of the support of
$N$. Given $N \in \fN$ with such a representation
we define a unique element of $\bM$ by
\begin{displaymath}
  \Sigma(N):= \bigboxplus_k \cY_k^{\boxplus m_k}\,,
\end{displaymath}
if the sum converges (recall from
Proposition~\ref{P:convergence_facts}(e) that the convergence of the
sum is independent of the order summands).  Thus, $\Sigma(\Psi(\cX)) =
\cX$ for all $\cX \in \bM$.

It is possible to topologize $\fN$ with the metrizable $w^\#$-topology
of \cite[Section A2.6]{MR1950431}.  This topology is the topology
generated by integration against bounded continuous functions that are
supported outside a neighborhood of $\cE$.  The resulting Borel
$\sigma$-field coincides with the $\sigma$-field generated by the
$\bN$-valued maps $N \mapsto N(B)$ Borel measurable, where $B$ is a
Borel subset of $\bM$ that is disjoint from some neighborhood of
$\cE$, see \cite[Theorem A2.6.III]{MR1950431}.

\begin{proposition}
  \label{P:factorization_measurable}
  The map $\Psi: \bM \to \fN$ is Borel measurable.
\end{proposition}

\begin{proof}
  The set $\{(\cX, \cY) \in \bM^2 : \cY \le \cX\}$ is closed by 
  Corollary~\ref{C:partial_order_closed}(a) and the set $\bI$ is
  $G_\delta$ by Proposition~\ref{P:irreducible_dense}.  It follows
  that the set $\bB := \{(\cX, \cY) \in \bM^2 : \cY \le \cX, \, \cY
  \in \bI\}$ is a $G_\delta$ subset of $\bM^2$ and, in particular, it
  is Borel.

For any $\cX \in \bM$, the section 
$\bB_\cX := \{\cY \in \bM: (\cX, \cY) \in \bB\}
= \{\cY \in \bM : \cY \le \cX, \, \cY \in \bI\}$ is countable
(indeed, it is discrete with $\cE$ as its only possible accumulation point).

By \cite[Exercise 18.15]{MR1321597}, the sets 
$\bT_n := \{\cX \in \bM : \# \bB_\cX = n\}$, $n=1,2,\ldots,\infty$, are
Borel and for each $n$ there exist Borel functions
$(\theta_i^{(n)})_{0 \le i < n}$ such that:
\begin{itemize}
\item
$\theta_i^{(n)} : \bT_n \to \bM$,
\item
the sets $\{(\cX, \cY) : \cX \in \bT_n, \, \cY = \theta_i^{(n)}(\cX)\}$,
$0 \le i < n$, $n=1,2,\ldots,\infty$, are pairwise disjoint,
\item
$\bB_\cX = \{\theta_i^{(n)}(\cX) : 0 \le i < n\}$ for $\cX \in \bT_n$,
$n=1,2,\ldots,\infty$. 
\end{itemize}

Recall the Borel function $M$ from 
Corollary~\ref{C:partial_order_closed}(b).  For $\cX \in \bT_n$, the set
$\{(\theta_i^{(n)}(\cX), M(\cX, \theta_i^{(n)}(\cX)): 0 \le i < n\}$
is a listing of the elements of the set $\{\cY \in \bI : \cY \le
\cX\}$ along with their multiplicities in the prime factorization of
$\cX$.  The functions $\cX \mapsto (\theta_i^{(n)}(\cX), M(\cX,
\theta_i^{(n)}(\cX))$, $\cX \in \bT_n$, $0 \le i < n$, 
$n=1,2,\dots,\infty$, are measurable and so  
\begin{displaymath}
  \cX \mapsto \Psi(\cX) =\sum_{i=0}^n  M(\cX,
  \theta_i^{(n)}(\cX)) \delta_{\theta_i^{(n)}(\cX)}
\end{displaymath}
for $\cX\in \bT_n$, provides a measurable map from $\bM$ to $\fN$, see
\cite[Proposition~9.1.X]{MR2371524}.
\end{proof}

\begin{remark}
  The map $\Psi$ is not continuous
  for the  $w^\#$-topology.  In fact, any
  $\cX\in (\bM \setminus \bI) \setminus\{\cE\}$ 
  is a discontinuity point, as the following
  argument demonstrates. Because $\bI$ 
  is dense in $\cK$, it is possible to find a
  sequence $\cX_n\in\bI$ that converges to $\cX$. Therefore,
  $\Psi(\cX_n) = \delta_{\cX_n}$, whereas $\Psi(\cX)$ has
  total mass at least two and the distance between
  any atom of $\Psi(\cX)$ and the point
  $\cX_n$ is bounded away from zero uniformly in $n$.
  % Therefore, the weak convergence of $\bM$-valued random elements is
  % not immediately related to the weak convergence of the corresponding
  % point processes on $\bI$.
\end{remark}

We omit the straightforward proof of the next result.

\begin{lemma}
\label{L:sum_measurable}
The set $\{N \in \fN : \Sigma(N) \; \text{is defined}\}$ is measurable and the
restriction of the map $\Sigma$ to this set is measurable.
\end{lemma}

\section{Scaling}
\label{S:scaling}

Given $\cX \in \bM$ and $a > 0$, set $a \cX := (X, a r_X, \mu_X) \in
\bM$.  This {\em scaling} operation $(a,\cX)\mapsto a\cX$ is jointly
continuous by Lemma~\ref{L:Gromov_weak} and it satisfies the first
distributivity law
\begin{equation}
\label{eq:first_distributive}
a(\cX \boxplus \cY) = (a \cX)
\boxplus (a \cY) \quad \text{for $\cX, \cY \in \bM$ and $a>0$.}
\end{equation}

The semigroup $(\bM,\boxplus)$ equipped with this scaling operation is
a convex cone. The neutral element $\cE$ is the origin in this cone;
that is, $\lim_{a \downarrow 0} a \cX = \cE$ for all $\cX \in \bM$,
which follows from Lemma~\ref{L:Gromov_weak}.  Note that $\diam(a
\cX) = a \diam(\cX)$ for $\cX \in \bM$ and $a > 0$.  

It is immediate from \eqref{eq:first_distributive} that $\cY \in \bI$ 
if and only if $a\cY \in \bI$ for all $a>0$

\begin{remark}
  \label{R:no_scaling}
  The Gromov--Prohorov metric is not homogeneous for this scaling operation;
   that is, $\dgpr(a \cX, a \cY)$ is not generally equal to
  $a \dgpr(\cX, \cY)$ for $a > 0$ and $\cX, \cY \in \bM$.  Moreover, it is not
  possible to equip $\bM$ with a homogeneous metric
  that induces the same topology as $\dgpr$. To see that this is so, 
  first note that for each $n\geq2$ there exists
   $\cX_n\in\bM \in \{\cE\}$ such that $\dgpr(c\cX_n,\cE)\leq
  n^{-1}$ for all $c > 0$; for example, take $\cX_n$ to be a
  two-point space with unit distance between the points and respective
  masses $n^{-1}$ and $1-n^{-1}$. For any sequence $(c_n)_{n\in\bN}$,
  we have $\dgpr(c_n\cX_n,\cE) \to 0$, while if $\delta$ is a homogeneous
  metric, then
  $\delta(c_n\cX_n,\cE) = \delta(c_n\cX_n, c_n \cE) 
  = c_n\delta(\cX_n,\cE)$ does not converge to
  zero if $c_n\to\infty$ sufficiently rapidly. 
\end{remark}

. 

% \query{scaling of semicharacters probably not needed}

% There is an analogue of the scaling operation for the 
% semigroup of semicharacters $(\chi_A)_{A \in \bA}$
% given by 
% \[
%   %\label{eq:scaling-chi}
%   a \chi_A(\cX)
%   :=\chi_A(a \cX)
%   =\chi_{a A}(\cX), 
%   \quad a>0, \, A \in \bA, \, \cX \in \bM.
% \]

We have seen that $(\bM,\le)$ is a distributive lattice.  There is a
large literature on lattices that are equipped with an action of the
additive group of the real numbers (see, for example \cite{MR0026240,
  MR0107613, MR0269560}).  Using exponential and logarithms to go back
and forth from one setting to the other, this work can be recast as
being about lattices with an action of the group consisting of
$\bR_{++}=(0,\infty)$ equipped with the usual multiplication of real
numbers.  Unfortunately, one of the hypotheses usually assumed in this
area translates to our setting as an assumption that $\cX < a \cX$ for
$a > 1$.  The following result shows that this is far from being the
case and also that scaling operation certainly does not satisfy the
second distributivity law.

\begin{proposition}
  \label{P:sd}
  Let $\cX$ be a metric measure space.
  \begin{itemize}
  \item[a)] If $\cX \le a \cX$ for some $a \ne 1$, then $a>1$ and $\cX
    = \bigboxplus_{k=1}^\infty a^{-k} \cZ$, where $\cZ$ is defined by
    the requirement that $a \cX = \cX \boxplus \cZ$.
  \item[b)] If $(a \cX) \boxplus (b \cX) = c \cX$, for some $a,b,c >
    0$, then $\cX=\cE$.
  \end{itemize}
\end{proposition}

\begin{proof}
  (a) Suppose that $\cX \ne \cE$ is such that $\cX \le a \cX$ for $a
  \ne 1$.  Recall the function $\Deltar(\cX)$ from
  \eqref{Eq:Delta}. Because $\Deltar(\cX) \le \Deltar(a \cX)$ and
  $\Deltar(a\cX)$ is monotone as function of $a\in\bR_+$, it must be
  the case that $a > 1$.  
  %It follows from Proposition~\ref{P:cancellation} that $\cZ$ well-defined.  
  We have $\cX = a^{-1} \cZ \boxplus a^{-1} \cX$.  Iterating, we have
  $\cX = \bigboxplus_{k=1}^n a^{-k} \cZ \boxplus a^{-n} \cX$. Since
  $\chi_1(a^{-n}\cX)\to1$, we have $a^{-n} \cX\to\cE$ by
  Lemma~\ref{L:Gromov_weak}.  By Proposition~\ref{P:cancellation}(b)
  or Proposition~\ref{P:convergence_facts}(c),
  $\lim_{n\to\infty}\bigboxplus_{k=1}^n a^{-k} \cZ$ exists.

  (b) Suppose that $(a \cX) \boxplus (b \cX) = c \cX$ for some $a,b,c
  > 0$.  Since $\Deltar(c\cX)=\Deltar((a\cX)\boxplus (b\cX))\geq
  (\Deltar(a\cX)\vee\Deltar(b\cX))$, we have $a\vee b\leq c$.
  % By part (a) of Lemma~\ref{lemma:diam}, $(a+b) \diam(\cX) =
  % \diam((a \cX) \boxplus (b \cX)) = \diam(c \cX) = c \diam(\cX)$, and
  % so $a+b = c$.  
  An irreducible element $\cY \in \bI$ appears in the factorization of
  $\cX$ guaranteed by Theorem~\ref{T:exist_unique_factorization} if
  and only if $c \cY \in \bI$ appears in the factorization of $c \cX$,
  and similar remarks hold for the factorizations of $a \cX$ and $b
  \cX$. Then $c\cY\leq a\cX$ or $c\cY\leq b\cX$. Assume the first, so
  that $\frac{c}{a}\cY\leq \cX$, so that $\frac{c}{a}\cY$ appears in
  the factorization of $\cX$. Iteration yields $(\frac{c}{a})^n\cY\leq
  \cX$ for all $n\geq1$, so that the spaces
  $((\frac{c}{a})^n\cY)_{n\in\bN}$ all belong to the prime
  decomposition of $\cX$ which then diverges by
  Proposition~\ref{P:convergence_facts}(g).
  % Suppose that
  % $\delta$ is the maximum of the diameters of the irreducible elements
  % that appear in the factorization of $\cX$.  (There could, of course,
  % be more than one - but only finitely many - irreducible factors with
  % maximal diameter.)  It follows that $c \cX$ has in irreducible
  % factor with diameter $c \delta$, whereas all the irreducible factors
  % of $(a \cX) \boxplus (b \cX)$ have diameters at most $(a \vee b)
  % \delta$, which is impossible since $(a \vee b) < c$.
\end{proof}

% \begin{remark}
%   An alternative proof of part (b) relies on
%   Corollary~\ref{C:no_additive} applied to the function
%   $\Phi(s)=s\cX$. 
% {\tt I don't see the following.  Part (b) says we can't have something happen
% for a given fixed $a,b,c$, whereas Corollary~\ref{C:no_additive} is about
% something happening for all $s$ and $t$.}  
%   Furthermore, similar arguments show that it is not
%   possible to fulfill neither $a\cX\boxplus b\cX\leq c\cX$ for all
%   $a,b,c>0$ nor $a\cX\boxplus b\cX\geq c\cX$ for all $a,b,c>0$ unless
%   $\cX=\cE$.
% {\tt Don't we want to know that we can't have these situations for a given
% $a,b,c$?  We can certainly have $a\cX\boxplus b\cX\geq c\cX$, just take
% $a = c$.  I don't see the argument that we can't have 
% $a\cX\boxplus b\cX\leq c\cX$ for some given $a,b,c$.}
% \end{remark}

\begin{remark}
  While it is possible to introduce a notion of convexity for subsets
  of $\bM$ using the addition and scaling in an obvious way, the
  absence of the second distributivity law makes the situation
  entirely different from the vector space case. For instance, a single
  point $\{\cX\}$ is not convex for $\cX\neq\cE$ and its convex hull
  is the set of spaces of the form $a_1\cX \boxplus \cdots \boxplus
  a_n\cX$ for $a_1,\dots,a_n\geq 0$ such that $a_1+\cdots+a_n=1$.  It
  is a consequence of Remark~\ref{R:no_LLN} for
  $a_1=\cdots=a_n=n^{-1}$ that this latter set is not even
  pre-compact.
\end{remark}

\begin{remark}
The map that sends $a \in \bR_{++}$ to the automorphism $\cX \mapsto a \cX$
of $(\bM, \boxplus)$ is a homomorphism from $(\bR_{++}, \times)$ to the group
of automorphisms of $(\bM,\boxplus)$.  We can therefore define the
{\em semidirect product} $\bM \rtimes \bR_{++}$ to be the semigroup
consisting of the set $\bM \times \bR_{++}$ equipped with the operation
$\boxasterisk$ defined by
\[
(\cX, a) \boxasterisk (\cY,b) := (\cX \boxplus (a \cY), a b).
\]
This semigroup has the identity element $(\cE,1)$ and is noncommutative.
The semidirect product of the group $(\bG, \boxplus)$ considered in
Remark~\ref{R:group_embedding}
% and Remark~\ref{R:group_embedding_redux}
and the group $(\bR_{++}, \times)$ can be defined similarly.
It would be interesting to extend the investigation of infinite divisibility
in Section~\ref{sec:infin-divis-laws} to this semigroup and group, but
we leave this topic for future study.
\end{remark}

\section{The Laplace transform}
\label{sec:rand-metr-meas}

A random element in $\bM$ is defined with respect to the Borel
$\sigma$-algebra on $\bM$ generated by the Gromov--Prohorov metric.

\begin{lemma}
\label{L:Laplace_transf_random}
Two $\bM$-valued random elements $\XX$ and $\YY$ have the same distribution
if and only if $\bE[\chi_A(\XX)] = \bE[\chi_A(\YY)]$ for all $A \in \bA$.
\end{lemma}

\begin{proof}
By Lemma~\ref{L:Gromov_weak}, the set of
functions $\{\chi_A : A \in \bA\}$ generates the Borel $\sigma$-algebra
on $\bM$.  From Remark~\ref{R:semicharacter_semigroup}, this set
is a semigroup under the usual multiplication of functions and, in particular,
it is closed under multiplication.  The result now follows from a standard
monotone class argument. 
\end{proof}

\begin{remark}
Recall from Section~\ref{S:factorization_measure}
the set $\fN$ of $\bN$-valued measures that are
concentrated on $\bI$ and the associated measurable structure.
Following the usual terminology, we define a {\em point process} 
to be a random element of $\fN$.
By Proposition~\ref{P:factorization_measurable},
any  $\bM$-valued
random element $\XX$ can, 
in the notation of Section~\ref{S:factorization_measure}, 
be viewed as a point process
$\NN := \Psi(\XX)$ such that $\Sigma(\NN) = \XX$.
If we write $\NN =\sum m_k\delta_{\YY_k}$ on $\bI$, then
  \begin{displaymath}
    \bE[\chi_A(\XX)]
    =\bE[\chi_A(\Sigma(\Psi(\XX))]
    =\bE\left[\prod \chi_A(\YY_k)^{m_k}\right]\,.
  \end{displaymath}
  The right-hand side is the expected value of the product of
  the function $\chi_A$ applied to each of
  the atoms of $\NN$ taking into account
  their multiplicities and hence it is an instance 
  of the {\em probability generating
  functional} of the point process $\NN$, see
  \cite[Equation~(9.4.13)]{MR2371524}. 
\end{remark}

\begin{remark}
\label{R:no_LLN}
  A fairly immediate consequence of Lemma~\ref{L:Laplace_transf_random}
  is that there is no analogue of a law of large numbers for random elements
  of $\bM$ in the sense that if 
  $(\XX_k)_{k\in\bN}$ is an i.i.d. sequence of random elements of $\bM$ 
  that are not identically
  equal $\cE$, then $\frac{1}{n} \bigboxplus_{k=0}^{n-1} \XX_k$ does
  not even have a subsequence that converges in distribution. 
  Indeed, for $A \in \bA$ with $A \in \bR_+^{\binom{m}{2}}$ we have
  \[
  \begin{split}
  & \lim_{n \to \infty} 
  \bE\left[\chi_A \left(\frac{1}{n} 
  \bigboxplus_{k=0}^{n-1} \XX_k \right)\right] 
  =
  \lim_{n \to \infty} 
  \biggl(\bE\left[\chi_A\left(\frac{1}{n}\XX_1\right)\right]\biggr)^n \\
  % & \quad =
  % \lim_{n \to \infty} 
  % \bE\left[\chi_{\frac{1}{n}A}\left(\XX_1\right)\right]^n \\
  & \quad =
  \lim_{n \to \infty}
  \biggl(
  \int_\bM \int_{X^m} 
  \exp\left(-\frac{1}{n} \sum_{1 \le i < j \le m} a_{ij} r_X(x_i,x_j)\right) 
  \\  & \qquad \times \, 
  \mu_X^{\otimes m}(dx) \, \bP\{\XX_1 \in d\cX\}
  \biggr)^n \\
  &\quad =
  \exp\bigg(-\lim_{n\to\infty} n\bigg(1-\int_\bM \int_{X^m} 
  \exp\left(-\frac{1}{n} \sum_{1 \le i < j \le m} a_{ij} r_X(x_i,x_j)
  \right)\\
  &\qquad \times \mu_X^{\otimes m}(dx) \, \bP\{\XX_1 \in d\cX\}\bigg)
  \bigg) \\
  & \quad =
  \exp\left(-\int_\bM \int_{X^m} 
  \sum_{1 \le i < j \le m} a_{ij} r_X(x_i,x_j)
  \, \mu_X^{\otimes m}(dx) \, \bP\{\XX_1 \in d\cX\}
  \right) \\
  & \quad =
  \exp\left(-\sum_{1 \le i < j \le m} a_{ij} \int_\bM \int_{X^2} 
   r_X(x_1,x_2)
  \, \mu_X^{\otimes 2}(dx) \, \bP\{\XX_1 \in d\cX\}
  \right). \\
  \end{split}
  \]
  % by a standard argument that is used to prove the 
  % weak law of large numbers for a sequence of 
  % i.i.d. nonnegative random variables using Laplace
  % transforms.  
  If some subsequence of $\frac{1}{n} \bigboxplus_{k=0}^{n-1} \XX_k$
  converged in distribution to a limit $\YY$, then we would have
  \[
  \begin{split}
  & \int_\bM \int_{Y^m} 
  \exp\left(-\sum_{1 \le i < j \le m} a_{ij} r_Y(y_i,y_j)\right)
  \, \mu_Y^{\otimes m}(dy) \, \bP\{\YY \in d\cY\} \\
  & \quad =
  \exp\left(-\sum_{1 \le i < j \le m} a_{ij} \int_\bM \int_{X^2} 
   r_X(x_1,x_2)
  \, \mu_X^{\otimes 2}(dx) \, \bP\{\XX_1 \in d\cX\}
  \right). \\
  \end{split}
  \]
  The right-hand side is the exponential of a linear combination of
  $a_{ij}$ and so corresponds to the Laplace transform of a
  deterministic random vector.  By the unicity of Laplace transforms
  for nonnegative random vectors, this implies that
  \[
  \begin{split}
  & \int_\bM \mu_Y^{\otimes 2}
  \biggl\{(y_1,y_2) \in Y^2 : r_Y(y_1,y_2) \\
  & \qquad\qquad\qquad \ne 
  \int_{X^2} 
   r_X(x_1,x_2)
  \, \mu_X^{\otimes 2}(dx) \, \bP\{\XX_1 \in d\cX\}\biggr\}
  \, \bP\{\YY \in d\cY\} \\
  & \quad = 0, \\
  \end{split}
  \]
  and hence there is a constant $c>0$ such that for $\bP\{\YY \in
  \cdot\}$-almost all $Y \in \bM$ we have $r_Y(y_1,y_2)=c$ for
  $\mu_Y^{\otimes 2}$-almost all $(y_1,y_2)\in Y^2$, but this is
  impossible for a nontrivial metric space $(Y,r_Y)$ and probability
  measure $\mu_Y$ with full support.
\end{remark}

%If $\XX$ is an $\bM$-valued random element, then the distribution of
%$\XX$ is determined by the Laplace transform $\bE[\chi_A(\XX)]$ for
%all semicharacters $\chi$. By \cite[Theorem~5.4]{dav:mol:zuy08}, this
%follows from the fact that these semicharacters form a separating
%family and generate the Borel $\sigma$-algebra on $\bM$. The latter is
%true, since the Gromov weak and Gromov--Prohorov topologies coincide,
%see \cite[Section~9]{grev:pfaf:win09}.

\section{Infinitely divisible random elements}
\label{sec:infin-divis-laws}

A random element $\YY$ of $\bM$ is {\em infinitely divisible} if for each
positive integer $n$ there are i.i.d. random elements 
$\YY_{n1}, \ldots, \YY_{nn}$ such that $\YY$ has the same distribution
as $\bigboxplus_{k=1}^n \YY_{nk}$. 

An $\bM$-valued {\em L\'evy process} is a $\bM$-valued stochastic process
$(\XX_t)_{t \ge 0}$ such that:
\begin{itemize}
\item
$\XX_0 = \cE$;
\item
$t \mapsto \XX_t$ is c\`adl\`ag (that is, right-continuous with
left-limits);
\item
given $0 = t_0 < t_1 < \ldots < t_n$,
there are independent $\bM$-valued random variables 
$\ZZ_{t_0 t_1}, \ZZ_{t_1 t_2}, \ldots, \ZZ_{t_{n-1} t_n}$ such that
the distribution of $\ZZ_{t_m t_{m+1}}$ only depends on $t_{m+1}-t_m$ for
$0 \le m \le n-1$ and
$\XX_{t_\ell} 
= 
\XX_{t_k} \boxplus\ZZ_{t_k t_{k+1}} \boxplus \cdots \boxplus \ZZ_{t_{\ell-1} t_\ell}$ 
for $0 \le k < \ell \le n$.
\end{itemize}

An account of the general theory of infinitely divisible distributions
on commutative semigroups may be found in \cite{ber:c:r}.  The
following result is the analogue in our setting of the classical
L\'evy--Hin\u{c}in--It\^o description of an infinitely divisible,
real-valued random variable.

\begin{theorem}
\label{T:Levy-Hincin-Ito}
\begin{itemize}
\item[a)] 
A random element $\YY$ of $\bM$ is infinitely divisible if and only if
it has the same distribution as $\XX_1$, where $(\XX_t)_{t \ge 0}$
is a L\'evy process with distribution uniquely specified by that of $\YY$.
\item[b)] 
For each $t > 0$ there is a unique random element $\Delta \XX_t$
such that $\XX_t = \XX_{t-} \boxplus \Delta \XX_t$.
\item[c)]
For each $t > 0$, $\XX_t = \bigboxplus_{0< s \le t} \Delta \XX_s$,
where the sum is a well-defined limit that
does not depend on the order of the summands.
\item[d)] 
The set of points
$\{(t,\Delta \XX_t):\; \Delta\XX_t \ne \cE\}$ form a Poisson point process on
$\bR_+ \times (\bM \setminus \{\cE\})$ with intensity measure
$\lambda \otimes \nu$, where $\lambda$ is Lebesgue measure and $\nu$
is a $\sigma$-finite measure on $\bM \setminus \{\cE\}$ such that 
\begin{equation}
  \label{eq:levy-measure}
  \int (D(\cX) \wedge 1) \, \nu(d\cX) < \infty\,.
\end{equation}
\item[e)] Conversely, if $\nu$ is a $\sigma$-finite measure on $\bM
  \setminus \{\cE\}$ satisfying \eqref{eq:levy-measure}, then there is
  an infinitely divisible random element $\YY$ and a L\'evy process
  $(\XX_t)_{t \ge 0}$ such that (a)-(d) hold, and the distributions of
  this random element and L\'evy process are unique.
\end{itemize}
\end{theorem}

\begin{proof}
Write $\bD$ for the set of nonnegative dyadic rational numbers.  It follows
from the infinite divisibility of $\YY$ and the Kolmogorov extension theorem
that we can build a family of random variables $(\XX_q)_{q \in \bD}$ such that:
\begin{itemize}
\item
$\XX_0 = \cE$,
\item
$\XX_1$ has the same distribution as $\YY$,
\item
Given $q_0, \ldots, q_n \in \bD$ with $0 = q_0 < q_1 < \ldots < q_n$,
there are independent $\bM$-valued random variables 
$\ZZ_{q_0 q_1}, \ZZ_{q_1 q_2}, \ldots, \ZZ_{q_{n-1} q_n}$ such that
the distribution of $\ZZ_{q_m q_{m+1}}$ only depends on $q_{m+1}-q_m$ for
$0 \le m \le n-1$ and
$\XX_{q_\ell} 
= 
\XX_{q_k} \boxplus\ZZ_{q_k q_{k+1}} \boxplus \cdots \boxplus \ZZ_{q_{\ell-1} q_\ell}$ for $0 \le k < \ell \le n$.  In particular, $\XX_p \le \XX_q$
for $p,q \in \bD$ with $p \le q$.
\end{itemize}

We claim that if $p \in \bD$, then
\begin{equation}
\label{E:right_cts_D}
\lim_{q \downarrow p, \, q \in \bD} \XX_q = \XX_p, \quad \text{a.s.}
\end{equation}
To see that this is the case, note that if $p,q \in \bD$ with $p < q$, then
$\XX_q = \XX_p \boxplus \ZZ_{pq}$ and it suffices to show that
$\lim_{q \downarrow p, \, q \in \bD} \dgpr(\ZZ_{pq},\cE) = 0$ almost surely. 

By Lemma~\ref{L:Gromov_weak}, it will certainly suffice to show that
$\lim_{q \downarrow p, \, q \in \bD} D(\ZZ_{p q}) = 0$ a.s.  However,
note that if we set $T_0 = 0$ and $T_r = D(\ZZ_{p, p+r})$ for $r \in
\bD \setminus \{0\}$, then the $\bR_+$-valued process $(T_r)_{r \in
  \bD}$ has stationary independent increments. It is well-known that
such a process has a c\`adl\`ag extension to the index set $\bR_+$ and
hence, in particular, $\lim_{r \downarrow 0, \, r \in \bD} T_r = 0$.

Lemma~\ref{L:path_clean-up} applied to $(\XX_p)_{p \in \bD}$ gives that 
it is possible to extend $(\XX_p)_{p \in \bD}$ to a L\'evy process
$(\XX_t)_{t \ge 0}$.  This establishes (a).  Moreover,  
for each $t > 0$ there is a unique $\bM$-valued random variable $\Delta \XX_t$
such that $\XX_t = \XX_{t-} \boxplus \Delta \XX_t$,
%$\sum_{0 < s \le t} \Delta \XX_s$ converges, 
and $\XX_t = \bigboxplus_{0 < s \le t} \Delta \XX_s$, where the sum is
well-defined by Proposition~\ref{P:convergence_facts}(e).  This
establishes (b) and (c).

A standard argument (see, for example, \cite[Theorem 12.10]{MR1876169})
shows that the set of points
$\{(t,\Delta \XX_t): \Delta\XX_t \ne \cE\}$ form a Poisson point process on
$\bR_+ \times (\bM \setminus \{\cE\})$.  The stationarity of the
``increments'' of $(\XX_t)_{t \ge 0}$ forces the intensity measure of
this Poisson point process to be of the form $\lambda \otimes \nu$,
and the fact that $\sum_{0 < s \le t} D(\Delta \XX_s)$ is
finite for all $t \ge 0$ implies \eqref{eq:levy-measure}, see, for example,
\cite[Corollary 12.11]{MR1876169}. This establishes (d).

We omit the straightforward proof of (e).
\end{proof}

Following the usual terminology, we refer to the $\sigma$-finite
measure $\nu$ in Theorem~\ref{T:Levy-Hincin-Ito} as the L\'evy measure
of the infinitely divisible random element $\YY$ or the L\'evy process
$(\XX_t)_{t \ge 0}$.  The next result is immediate from
Theorem~\ref{T:Levy-Hincin-Ito}, the multiplicative property of the
semicharacters $\chi_A$, and the usual formula for the Laplace
functional of a Poisson process.

\begin{corollary}
\label{C:Laplace_Levy_Hincin}
If $\YY$ is an infinitely divisible random element of $\bM$
with L\'evy measure $\nu$, then the Laplace transform
of $\YY$ is given by
\begin{equation}
  \label{eq:laplace-id}
  \bE[\chi_A(\YY)] = \exp\left(-\int (1 - \chi_A(\cY)) \, \nu(d\cY)\right), 
  \quad A \in \bA.
\end{equation}
\end{corollary}

\begin{remark}
  In the notation of Theorem~\ref{T:Levy-Hincin-Ito},
  the random measure 
  \[
  \sum_{0 < t \le 1} \delta_{\Delta\XX_t}
  \] 
  is a Poisson random measure on $\bM$ with intensity measure $\nu$
  and we have $\YY = \XX_1 = \bigboxplus_{0 < t \le 1} \Delta\XX_t$.
  The push-forward of this random measure by the map $\Psi$ of
  Proposition~\ref{P:factorization_measurable} is a Poisson random
  measure on the space $\fN$ of $\bN$-valued measures that are
  concentrated on $\bI$. The intensity measure of this latter Poisson
  random measure is the push-forward $Q$ of the L\'evy measure $\nu$
  by $\Psi$.  The ``points'' of the latter Poisson random measure are
  usually called {\em clusters} in the point processes literature,
  while $Q$ itself is called the {\em KLM measure}, see
  \cite[Definition~10.2.IV]{MR2371524}. Let $\NN$ be the point process
  on $\bI$ obtained as the superposition of clusters; that is, $\NN =
  \sum_{0 < t \le 1} \Psi(\Delta \XX_t)$ is the sum of the
  $\bN$-valued measures given by each individual cluster.  This point
  process on $\bI$ is called the {\em Poisson cluster process} in the
  Poisson point process literature.  The infinite divisibility of
  $\YY$ implies the infinite divisibility of the point process
  $\Psi(\YY)$ and the equality $\Psi(\YY) = \NN$ is an instance of the
  well-known fact that infinitely divisible point processes are
  Poisson cluster processes. Furthermore, \eqref{eq:laplace-id}
  corresponds to the classical representation of the probability
  generating functional of an infinitely divisible point process
  specialized to the space $\bI$, see
  \cite[Theorem~10.2.V]{MR2371524}. On the other hand, if $\MM$ is a
  Poisson cluster process on $\bI$ such that $\Sigma(\MM)$ is almost
  surely well-defined, then $\Sigma(\MM)$ is an infinitely divisible
  random element of $\bM$, and our observations above show that all
  infinitely divisible random elements of $\bM$ appear this way.
\end{remark}

% Inspired by the concepts from the theory of point processes, an
% infinitely divisible random element $\XX$ is said to be regular if its
% L\'evy measure $\nu$ is supported by metric measure spaces whose prime
% factorization is finite and is singular if $\nu$ is supported by
% metric measure spaces with infinite factorizations. Each infinitely
% divisible $\XX$ can be represented as the sum of independent regular
% and singular components.

%The following result implies the absence of idempotent random elements
%in $\bM$.
%
%\begin{lemma}
%  \label{lemma:idemp}
%  If $\XX'\boxplus\XX''$ equals in distribution to $\XX$ for
%  i.i.d. $\XX,\XX',\XX''$, then $\XX=\cE$ a.s.
%\end{lemma}
%\begin{proof}
%  In view of Lemma~\ref{lemma:diam},
%  $\essdiam(\XX'\boxplus\XX'')=\essdiam(\XX')+\essdiam(\XX'')$ has the
%  same distribution as $\essdiam(\XX)$, which is only possible if
%  $\XX=\cE$ a.s. 
%\end{proof}

We end this section with a deterministic path-regularization
result that was used in the proof of Theorem~\ref{T:Levy-Hincin-Ito}.

\begin{lemma}
  \label{L:path_clean-up}
  Suppose that $\Xi : \bD \to \bM$ is such that $\Xi(0) = \cE$,
  $\Xi(p) \le \Xi(q)$ for $p,q \in \bD$ with $0 \le p \le q$, and
  $\lim_{q \downarrow p, \, q \in \bD} \Xi(q) = \Xi(p)$ for all $p \in
  \bD$.  Then, $\bar \Xi(t) := \lim_{q \downarrow t, \, q \in \bD}
  \Xi(q)$ exists for all $t \in \bR_+$.  Moreover, the function $\bar
  \Xi: \bR_+ \to \bM$ has the following properties:
  \begin{itemize}
  \item $\bar \Xi(p) = \Xi(p)$ for $p \in \bD$,
  \item $\bar \Xi(s) \le \bar \Xi(t)$ for $s,t \in \bR_+$ with $s \le
    t$,
  \item $t \mapsto \bar \Xi(t)$ is c\`adl\`ag,
  \item for $p,q \in \bD$ with $0 \le p < q$, there is a unique
    $\Theta(p,q) \in \bM$ such that $\Xi(q) = \Xi(p) \boxplus
    \Theta(p,q)$,
  \item for $0 \le s < t$, there is a unique $\bar \Theta(s,t) \in
    \bM$ such that $\bar \Xi(t) = \bar \Xi(s) \boxplus \bar
    \Theta(s,t)$ and $\bar \Theta(s,t) = \lim_{p \downarrow s, \, q
      \downarrow t, \, p,q \in \bD} \Theta(p,q),$
  \item for each $t > 0$ there is a unique $\Delta \bar \Xi(t) \in
    \bM$ such that $\bar \Xi(t) = \lim_{s \uparrow t} \bar \Xi(s)
    \boxplus \Delta \bar \Xi(t)$,
  \item $\sum_{u < t \le v} D(\Delta \bar \Xi(t)) 
  \le D(\bar \Theta(u,v))$ for all $0 \le u < v$,
  \item
  the sum $\bigboxplus_{0 < s \le t} \Delta \bar \Xi(s)$ is
  well-defined for all $t \ge 0$,
  \item
  $\bar \Xi(t) = \bigboxplus_{0 < s \le t} \Delta \bar \Xi(s)$ 
  for all $t \ge 0$.
  \end{itemize}
\end{lemma}

\begin{proof}
  It follows from Proposition~\ref{P:convergence_facts}(d) that
  $\lim_{q \downarrow t, \, q \in \bD} \Xi(q) =: \bar \Xi(t)$ exists
  for all $t \ge 0$.

  It is clear that $\bar \Xi(p) = \Xi(p)$ for $p \in \bD$ and that
  $\bar \Xi(s) \le \bar \Xi(t)$ for $s,t \in \bR_+$ with $s \le t$.
  It is also clear that $t \mapsto \bar \Xi(t)$ is right-continuous.
  By Proposition~\ref{P:convergence_facts}(c),
  $\bar \Xi(t-) := \lim_{s \uparrow t} \bar \Xi(s)$ exists for
  all $t > 0$ and $\bar \Xi(t-) \le \bar \Xi(t)$ for all $t > 0$.

  The existence and uniqueness of $\bar \Theta(s,t)$ such that $\bar
  \Xi(t) = \bar \Xi(s) \boxplus \bar \Theta(s,t)$ and the fact that
  $\bar \Theta(s,t) = \lim_{p \downarrow s, \, q \downarrow t, \, p,q
    \in \bD} \Theta(p,q)$ follow from
  Proposition~\ref{P:cancellation}(b).
 
  It is a consequence of Proposition~\ref{P:cancellation}(b) that
  $\Delta \bar \Xi(t)$ exists and is well-defined.

  For any $0 \le u < v$ and $u<t_1 < \cdots < t_n \le v$ we have
  $\Delta \bar \Xi(t_1) \boxplus \cdots \boxplus \Delta \bar \Xi(t_n)
  \le \bar \Theta(u,v)$. Hence, by
  Proposition~\ref{P:cancellation}(a), $\bigboxplus_{0 < s \le t}
  \Delta \bar \Xi(s)$ is well-defined.

It is clear that 
$\bigboxplus_{0 < s \le t} \Delta \bar \Xi(s) \le \bar \Xi(t)$
for all $t \ge 0$ and so we can use Proposition~\ref{P:cancellation} to define 
a unique function $\Phi: \bR_+ \to \bM$ such that 
$\bar \Xi(t) = \Phi(t) \boxplus \bigboxplus_{0 < s \le t} \Delta \bar \Xi(s)$
for all $t \ge 0$.  The function $\Phi$ is continuous and $\Phi(s) \le \Phi(t)$
for $0 \le s < t$. Also, $\Phi(0) = \bE$.   
Corollary~\ref{C:no_cts_increasing} gives
that $\Phi \equiv \cE$, completing the
proof of the lemma.
\end{proof}

\section{Stable random elements}
\label{sec:stable-metr-meas}

A $\bM$-valued random element $\YY$ is {\em stable} with index $\alpha
> 0$ if for any $a,b>0$ the random element
$(a+b)^{\frac{1}{\alpha}}\YY$ has the same distribution as
$a^{\frac{1}{\alpha}} \YY' \boxplus b^{\frac{1}{\alpha}} \YY''$, where
$\YY'$ and $\YY''$ are independent copies of $\YY$.  Note that a
stable random element is necessarily infinitely divisible. If $\YY$ is
stable and almost surely takes values in 
the space of bounded metric measure spaces, 
then its diameter is a nonnegative strictly stable
random variable.
  
There is a general investigation of stable random elements of convex
cones in \cite{dav:mol:zuy08}.  In general, not all such objects have
Laplace transforms that are of the type analogous to those described
in Corollary~\ref{C:Laplace_Levy_Hincin}.  For example, there can be
Gaussian-like distributions.  However, no such complexities arise in
our setting.
  
\begin{theorem}
\label{T:stable_characterization}
Suppose that $\YY$ is a nontrivial $\alpha$-stable random element of $\bM$.
Then, $0 < \alpha < 1$ and the L\'evy measure $\nu$ of $\YY$
obeys the scaling condition
\begin{equation}
\label{Eq:nu_scale}
  \nu(a B)=a^{-\alpha}\nu(B),\quad a>0,
\end{equation}
for all Borel sets $B \subseteq \bM$.  Conversely, if $\nu$ is a
$\sigma$-finite measure on $\bM \setminus \{\cE\}$ that obeys the
scaling condition for $0 < \alpha < 1$ and satisfies
\eqref{eq:levy-measure}, then $\nu$ is the L\'evy measure of an
$\alpha$-stable random element.
\end{theorem}

\begin{proof}
If $(\XX_t)_{t \ge 0}$ is the L\'evy process corresponding to $\YY$, 
then it is not difficult to check that the process 
$(a^{-\frac{1}{\alpha}} \XX_{a t})_{t \ge 0}$ has the same distribution
as $(\XX_t)_{t \ge 0}$, and the scaling condition for
$\nu$ follows easily.  
Since $r_\YY(\xi_1,\xi_2)$ is a nonnegative stable random variable of
index $\alpha$, we necessarily have $\alpha\in(0,1)$. 
% Write $\eta$ for the push-forward of $\nu$ by the map $\diam$.
% It is clear from the scaling condition and the property
% $\diam(a \cX) = a \diam(\cX)$ that $\eta([x,\infty)) = c x^{-\alpha}$
% for some constant $c > 0$, and so
% \[
% \int_{\bM\setminus\{\cE\}} (\diam(\cY) \wedge 1) \, \nu(d\cY)
% =
% \int_{(0,\infty)} (y \wedge 1) \, c \alpha y^{-(\alpha+1)} \, dy.
% \]
% In order for this integral to be finite, it must be the case that $\alpha < 1$.
The remainder of the proof is straightforward and we omit it.
\end{proof}

\begin{remark}
One of the conclusions of Theorem~\ref{T:stable_characterization} is that there
are no nontrivial $\alpha$-stable random elements for $\alpha \ge 1$.  
This is also a consequence of the following argument.
If $\YY$ was a nontrivial $\alpha$-stable random element and 
$(\YY_k)_{k \in \bN}$ was a sequence of independent copies of $\YY$, then
$n^{-\frac{1}{\alpha}} \bigboxplus_{k=0}^{n-1} \YY_k$ would have the
same distribution as $\YY$ 
and hence $\frac{1}{n} \bigboxplus_{k=0}^{n-1} \YY_k$
would certainly converge in distribution as $n \to \infty$, but
this contradicts Remark~\ref{R:no_LLN}, where we observed that there is
no analogue of a law of large numbers in our setting.
\end{remark}

We finish this section with an analogue of the classical LePage representation
of stable real random variables.  

\begin{theorem}
  \label{Th:LePage}
The following are equivalent for a random element $\YY$ of $\bM$.
\begin{itemize}
\item[a)]
The random element $\YY$ is $\alpha$-stable.
\item[b)]
The random element $\YY$ is infinitely divisible with a L\'evy measure
$\nu$ that is of the form 
$\nu(B) 
= \alpha \int_0^\infty \pi(t^{-1} B) \, t^{-(\alpha+1)} \, dt$
for all Borel sets $B \subseteq \bM \setminus \{\cE\}$,
where 
$\pi$ is a probability measure on $\bM \setminus \{\cE\}$
such that
\begin{equation}
    \label{Eq:pi_integrable}
    \int_{\bM \setminus \{\cE\}} \int_{Z^2} 
    r_Z^\alpha(x,y) \, \mu_Z^{\otimes 2}(dz_1,dz_2) \, \pi(d\cZ)<\infty,
  \end{equation}
In particular, $\pi$ assigns all of its mass to metric measure spaces
 $\cZ$ for which 
\begin{equation}
\label{Eq:distance_alpha_integrable}
\int_{Z^2} r_Z^\alpha(x,y) \, \mu_Z^{\otimes 2}(dz_1,dz_2) < \infty.
\end{equation}
\item[c)]
The random element $\YY$ has the same distribution as
  \begin{equation}
    \label{Eq:le-page-no-diam}
    \bigboxplus_{n \in \bN} \Gamma_n^{-\frac{1}{\alpha}} \ZZ_n,
  \end{equation}
  where $(\Gamma_n)_{n \in \bN}$ is
  the sequence of successive arrivals of a homogeneous, unit intensity
  Poisson point process on $\bR_+$ and $(\ZZ_n)_{n\in\bN}$ is a
  sequence of i.i.d. random elements in $\bM \setminus \{\cE\}$ with common
  distribution $\pi$ such that \eqref{Eq:pi_integrable} holds.
\end{itemize}
\end{theorem}

\begin{proof}
Suppose that $\YY$ is $\alpha$-stable with L\'evy measure $\nu$.
We know from Theorem~\ref{T:stable_characterization} that $\nu$
satisfies the scaling condition \eqref{Eq:nu_scale}
and the integrability condition \eqref{eq:levy-measure}.

For any $\cX \in \bM \setminus \{\cE\}$ 
the function $t \mapsto D(t\cX)$ is strictly increasing and
$\sup_{t > 0} D(t \cX) 
= - \log(r_X^{\otimes 2} \{(x_1,x_2) \in X^2 : x_1 = x_2\})$
with the convention $- \log(0) = \infty$.
Set 
\[
\bV_0 := \{\cX \in \bM \setminus \{\cE\} : \sup_{t > 0} D(t \cX) > 1\}
\] 
and 
\[
\bV_k := \{\cX \in \bM \setminus \{\cE\} : 2^{-k} < \sup_{t > 0} D(t \cX) \le 2^{-(k-1)}\}
\] 
for $k \ge 1$. The sets $\bV_k$, 
$k \in \bN$, are disjoint, 
their union is $\bM \setminus \{\cE\}$, and  $\cX \in \bV_k$ for some 
$k \in \bN$ if and only if $t \cX \in \bV_k$ for all $t \in \bR_{++}$.

Define $\tau : \bM \setminus \{\cE\} \rightarrow \bR_{++}$ as follows.
For $\cX \in \bV_k$, set 
$\tau(\cX) := \inf\{t > 0 : D(t^{-1} \cX) \le 2^{-k}\}$.
Observe that $\tau(s \cX) = s \tau(\cX)$ for all 
$\cX \in \bM \setminus \{\cE\}$  and $s \in \bR_{++}$ and that
$D(\tau(\cX)^{-1} \cX) = 2^{-k}$ for $\cX \in \bV_k$.
Note that if $\tau(\cX)^{-1} \cX = \tau(\cY)^{-1} \cY$, then 
$\{t \cX: t > 0\}  = \{t \cY: t > 0\}$, whereas if 
$\tau(\cX)^{-1} \cX \ne \tau(\cY)^{-1} \cY$, then 
$\{t \cX: t > 0\}  \cap \{t \cY: t > 0\} = \emptyset$.  In other words,
the set $\{\tau(\cX)^{-1} \cX : \cX \in \bM \setminus \{\cE\}\}$ 
is a cross-section of
orbit representatives for the action of the group $\bR_{++}$ on
$\bM \setminus \{\cE\}$.

For each $k \in \bN$,
the maps $\cX \mapsto (\tau(\cX)^{-1} \cX, \tau(\cX))$ and
$(\cY,t) \mapsto t \cY$ are mutually inverse Borel bijections between
the Borel sets $\bV_k$ and $\bS_k \times \bR_{++}$, where
$\bS_k := \{\cX \in \bV_k : \tau(\cX) = 1\}$.  Let
$\tilde \nu$ be the push-forward of $\nu$ by the map 
$\cX \mapsto (\tau(\cX)^{-1} \cX, \tau(\cX))$
and define a measure $\rho_k$ on
$\bS_k$ by $\rho_k(A) = \tilde \nu(A \times [1,\infty))$.
Since 
\[
\rho_k(\bS_k) 
= \tilde \nu(\bS_k \times [1,\infty))
\le \nu\{\cX \in \bM \setminus \{\cE\} : D(\cX) \ge 2^{-k}\},
\]
it follows from \eqref{eq:levy-measure} that the total mass of $\rho_k$
is finite.

The scaling property \eqref{Eq:nu_scale} of $\nu$ is equivalent to the
scaling property 
$\tilde \nu(A \times s B) 
= s^{-\alpha} \tilde \nu(A \times B)$
for $s \in \bR_{++}$ and Borel sets $A \subseteq \bS_k$ and
$B \subseteq \bR_{++}$.  Thus, if we let $\theta$ be
the measure on $\bR_{++}$ given by $\theta(dt) = \alpha t^{-(\alpha+1)} \, dt$,
then 
\[
\begin{split}
\tilde \nu(A \times [b,\infty)) 
& = \tilde \nu(A \times b [1,\infty)) \\
& = b^{-\alpha} \tilde \nu(A \times [1,\infty)) \\
& = \rho_k(A) \times \theta([b,\infty)) \\
\end{split}
\]
for $A \subseteq \bS_k$.
Therefore the restriction of 
$\tilde \nu$ to $\bS_k \times \bR_{++}$
is $\rho_k \otimes \theta$ and hence the restriction of
$\nu$ to $\bV_k$ is the push-forward of $\rho_k \otimes \theta$ by
the map $(\cY,t) \mapsto t \cY$.  

We can think of $\rho_k$ as a measure on all of $\bV_k$.  For
$c_k \in \bR_{++}$, let $\eta_k$ be the measure on $\bV_k$ that
assigns all of its mass to the set $c_k \bS_k$ and is given
by $\eta_k(A) = c_k^\alpha \rho_k(c_k^{-1} A)$.  We have
\[
\begin{split}
\eta_k \otimes \theta\{(\cY,t) : t \cY \in B\}
& = \int \eta_k(t^{-1} B)  \alpha t^{-(\alpha+1)} \, dt \\
& = \int c_k^\alpha \rho_k(c_k^{-1} t^{-1} B) \alpha t^{-(\alpha+1)} \, dt \\
& = \int \rho_k(s^{-1} B) \alpha s^{-(\alpha+1)} \, ds \\
& =  \rho_k \otimes \theta\{(\cY,t) : t \cY \in B\} \\
& = \nu(B), \\
\end{split}
\]
and so $\eta_k$ is a finite measure with total mass 
$c_k^\alpha \rho_k(\bS_k)$ that has the property that
the push-forward of $\eta_k \otimes \theta$ by
the map $(\cY,t) \mapsto t \cY$ is the restriction of
$\nu$ to $\bV_k$.

We can regard $\eta_k$ as being a
finite measure on all of $\bM \setminus \{\cE\}$ and,
by choosing the constants $c_k$, $k \in \bN$, appropriately
we can arrange for $\pi := \sum_{k \in \bN} \eta_k$ to be
a probability measure. We have
\[
\nu(B) 
= \alpha \int_0^\infty \pi(t^{-1} B) \, t^{-(\alpha+1)} \, dt
\]
for all Borel sets $B \subseteq \bM \setminus \{\cE\}$.

It follows from \eqref{eq:levy-measure} that 
\[
  \int_0^\infty \int_{\bM \setminus \{\cE\}} (D(t\cZ)\wedge 1) t^{-(\alpha+1)} 
    \, \pi(d\cZ) \, dt
    < \infty.
\]
By Lemma~\ref{L:chi1_chiA}, for any $\sigma$-finite measure $\lambda$ 
on $\bM$
the integral 
\[
\int_0^\infty \int_{\bM \setminus \{\cE\}} (D(t\cZ)\wedge 1) t^{-(\alpha+1)} 
    \, \lambda(d\cZ) \, dt
\]
is bounded above and below by constant multiples of   
  \begin{displaymath}
    \int_{\bM \setminus \{\cE\}}
    \int_{Z^2} ((tr_Z(z_1,z_2))\wedge 1) t^{-(\alpha+1)}
    \, \mu_Z^{\otimes 2}(dz_1,dz_2) \, \lambda(d\cZ) \, dt.
  \end{displaymath}
  The latter integral is a constant multiple of 
\[
\int_{\bM \setminus \{\cE\}} \int_{Z^2} 
    r_Z^\alpha(x,y) \, \mu_Z^{\otimes 2}(dz_1,dz_2) \, \lambda(d\cZ)
\]
  because
  \begin{displaymath}
    \int_0^\infty ((tr)\wedge 1) t^{-(\alpha+1)} \, dt 
    = \frac{1}{\alpha(1-\alpha)} r^\alpha.
  \end{displaymath}
for any $r \ge 0$.

This completes the proof that (a) implies (b).  The proof that (b) implies (a)
simply involves checking that the measure $\nu$ satisfies the scaling property
\eqref{Eq:nu_scale} and the integrability property \eqref{eq:levy-measure}.
The former is obvious and the latter follows from 
the argument immediately above.

The proof that (a) and (b) are equivalent to (c) requires showing that
$\nu$ is a measure satisfying the conditions of (b) if and only if
the points of a Poisson random measure on $\bM \setminus \{\cE\}$ with
intensity $\nu$ have the same distribution as the random set 
$( \Gamma_n^{-\frac{1}{\alpha}} \ZZ_n)_{n \in \bN}$.  However,
if $(\ZZ_n, \Gamma_n)_{n \in \bN}$ are as in (c), then they are the 
points of a Poisson random measure on $(\bM \setminus \{\cE\}) \times \bR_{++}$
with intensity $\pi \otimes \lambda$, where $\lambda$ is Lebesgue measure,
and so $(\ZZ_n, \Gamma_n^{-\frac{1}{\alpha}})_{n \in \bN}$
are the points of a Poisson random measure with intensity 
$\pi \otimes \theta$,
where the measure $\theta$ is as above.
\end{proof}

\begin{remark}
The probability measure $\pi$ in Theorem~\ref{Th:LePage}
is not unique.  However,
in the proof that (a) implied (b) the $\pi$ that was constructed
was concentrated on a set $\bT$ with the property that for all 
$\cX \in \bM \setminus \{\cE\}$ there is a unique $t \in \bR_{++}$
such that $t \cX \in \bT$.  If part (b) holds with a $\pi$ that is
supported on a set $\bU$ with this property, then $\pi$ is the unique
probability measure concentrated on $\bU$ that leads to a representation
of $\nu$ in the manner described in the theorem.
\end{remark}

\begin{remark}
It follows readily from Theorem~\ref{Th:LePage}
that a bounded metric measure space is $\alpha$-stable if and only if it
  admits a representation of the form \eqref{Eq:le-page-no-diam} where
  $(\ZZ_n)_{n\in\bN}$ is any sequence of i.i.d. random elements in
  $\bM$ such that $\diam(\ZZ_n)=\gamma$ almost surely for a suitable
  constant $\gamma$. An alternative proof
  of this fact can be carried out using \cite[Theorems~3.6
  and~7.14]{dav:mol:zuy08}.
\end{remark}

\begin{example}
  We can construct an $\alpha$-stable random element $\YY$ by
  considering the LePage series in which the $\ZZ_n$ are copies of
  some common nonrandom bounded metric measure space.
  In this case on the set of full probability where 
  $\sum_{n \in \bN} \Gamma_n^{-\frac{1}{\alpha}} < \infty$,
  $\YY$ is the infinite Cartesian product $Y:=Z^\infty$ 
  equipped with the random metric 
  \begin{displaymath}
    r_Y((z_n'),(z_n'')) 
    := \sum_{n \in \bN} \Gamma_n^{-\frac{1}{\alpha}} 
    r_Z(z_n',z_n'')
  \end{displaymath}
  and the probability measure $\mu_Y := \mu_Z^{\otimes \infty}$.
\end{example}

\section{Thinning}
\label{S:thinning}

Recall the map $\Psi$ that associates with each $\cX\in\bM$ an
$\bN$-valued measure on $\bI$.
% : \bM \to (\bI \cup \{\cE\})^\bN$ from
% Proposition~\ref{P:factorization_measurable} with the property that
% $\Psi(\cX)_n$, $\Psi(\cX)_n \ne \cE$, $n \in \bN$, is a listing of the
% primes elements in the factorization of $\cX$ repeated according to
% their multiplicity.  Let $(\xi_n)_{n \in \bN}$ be a sequence of
% i.i.d. $\{0,1\}$-valued random variables with $\bP\{\xi_n = 1\} = p$
% for some $0 \le p \le 1$.
% Given a $\bM$-valued random element $\XX$
% that is independent of $(\xi_n)_{n \in \bN}$, set
% \[
% \XX^{(p)} := \bigboxplus_{n \in \bN} \xi_n \Psi(\XX)_n.
% \]
For $p\in[0,1]$, the independent $p$-thinning of an $\bN$-valued
measure $N:=\sum_k m_k\delta_{\cY_k}$ is defined 
in the usual way as 
$N^{(p)}:=\sum_k \xi_k \delta_{\cY_k}$, 
where $\xi_k$, $k\in\bN$, are independent
binomial random variables with parameters $m_k$ and $p$. In other
words, each atom of $N$ is retained with probability $p$ and otherwise
eliminated independently of all other atoms and taking into account the
multiplicities. 

Applying an independent $p$-thinning procedure to the point process
$\NN:=\Psi(\XX)$ generated by random element $\XX$ in $\bM$ yields a
$\bM$-valued random element $\XX^{(p)}:=\Sigma(\NN^{(p)})$ that we call the
\emph{$p$-thinning} of $\XX$.  Note that the $\XX^{(p)} \le \XX$,
$\XX^{(0)} = \cE$, $\XX^{(1)} = \XX$, and for $0 \le p,q \le 1$ 
the random element $(\XX^{(p)})^{(q)}$ has the
same distribution as the random element $\XX^{(pq)}$.
It is possible
to build an $\bM$-valued strong Markov process $(\XX_t)_{t \ge 0}$ so that the
conditional distribution of $\XX_{s+t}$ given $\{\XX_s = \cX\}$ 
is the $e^{-t}$-thinning of $\cX$; each irreducible factor of
is equipped with an independent exponential random clock that has
expected value $1$ and the factor appears in the decomposition of
$\XX_t$ into irreducibles provided its clock has not rung by time $t$.

% Moreover, the distribution of $\XX^{(p)}$ only depends on that of
% $\XX$ and it is independent of the particular choice of the map $\Psi$
% satisfying the conditions of
% Proposition~\ref{P:factorization_measurable}.
% We call $\XX^{(p)}$ the {\em $p$-thinning} of $\XX$ because of its
% resemblance to the well-known thinning operation for point processes.

  Also, if $\XX$
and $\YY$ are independent random elements and $\XX^{(p)}$ and
$\YY^{(p)}$ are constructed to be independent, then $\XX^{(p)}
\boxplus \YY^{(p)}$ has the same distribution as $(\XX \boxplus
\YY)^{(p)}$.  It follows from this last property that, for fixed $A
\in \bA$ and $0 \le p \le 1$, the map
\[
\cX \mapsto \bE[\chi_A(\cX^{(p)})] 
= 
\prod \left(1 - p + p \chi_A(\cY_n)\right)
\]
is a semicharacter, where the product ranges over the factors that
appear in the factorization of $\cX$ into a sum of irreducible
elements of $\bM$ (repeated, of course, according to their
multiplicities).  This is a particular case of the construction in
Remark~\ref{R:factorization_semicharacter}.

The thinning operation can be used to construct
$\bM$-valued stochastic processes that are not necessarily
increasing or decreasing in the $\le$ partial order by combining
the $\boxplus$ addition of independent random increments with thinning; that is, the semigroup of the process is the Trotter product of the semigroup of
a L\'evy process and the semigroup of the Markov process introduced above that evolves in such a way that the value of the process at time $t$ is the $e^{-t}$-thinning of its value at time $0$.

Furthermore, the thinning procedure is the key ingredient for defining
a notion of discrete stability analogous to that in
\cite{MR2817615}. A random metric measure space $\XX$ is said to be
{\em discrete stable} of index $\alpha$ if $\XX$ coincides in distribution with
$\XX_1^{(t)^{1/\alpha}}\boxplus\XX_2^{(1-t)^{1/\alpha}}$, where
$\XX_1$ and $\XX_2$ are independent copies of $\XX$. By an application
of general results from \cite{MR2817615} it is possible to conclude
that such an $\XX$ corresponds to a doubly stochastic (Cox) Poisson
process on $\bI$ whose random intensity measure is stable. The
simplest example is $\XX:=\cY^{\boxplus N}$, where $\cY\in\bI$ and $N$
is an $\bN$-valued discrete $\alpha$-stable random variable; any such
random variable $N$ has a probability generating function of the form
$\bE [s^N]=\exp(-c(1-s)^\alpha)$, $s\in[0,1]$, where $c\in\bR_{++}$.

\section{The Gromov-Prohorov metric}
\label{S:metric_summary}

We follow the definition of the {\em Gromov-Prohorov} 
metric in \cite{grev:pfaf:win09}.

Recall that the distance in the {\em Prohorov metric} between two
probability measures $\mu_1$ and $\mu_2$ on a common metric space $(Z,
r_Z)$ is defined by
\[
\dpr^{(Z, r_Z)}(\mu_1,\mu_2) 
:= \inf\{\epsilon > 0 :  \mu_1(F) \le  \mu_2(F^\epsilon) +  \epsilon, 
\, \forall F \, \text{closed}\},
\]
where
\[
F^{\epsilon}  :=
\{z  \in Z : r_Z(z, z') <  \epsilon, \, \text{for some} \,  z' \in F\}.
\]
An alternative characterization of the Prohorov metric due to Strassen
(see, for example, \cite[Theorem 3.1.2]{MR838085} 
or \cite[Corollary 11.6.4]{MR1932358})
is that
\[
\dpr^{(Z, r_Z)}(\mu_1,\mu_2)
=
\inf_\pi 
\inf\{\epsilon > 0: 
\pi\{(z,z') \in Z \times Z : r_Z(z,z') \ge \epsilon\}
\le \epsilon\},
\]
where the infimum is over all probability measures $\pi$ such that 
$\pi(\cdot \times Z) = \mu_1$ and $\pi(Z \times \cdot) = \mu_2$.

The following result is no doubt well-known, but we
include it for completeness.  Recall that if
$(X, r_X)$ and $(Y, r_Y)$ are two metric spaces, then
$r_X \oplus r_Y$ is the metric on the Cartesian product $X \times Y$
given by 
$r_X \oplus r_Y((x',y'), (x'',y'')) = r_X(x',x'') + r_Y(y',y'')$.

\begin{lemma}
\label{L:Prohorov_product}
Suppose that $\mu_1$ and $\mu_2$ (resp. $\nu_1$ and $\nu_2$)
are probability measures on a metric space $(X, r_X)$
(resp. $(Y, r_Y)$).  Then,
\[
\dpr^{(X \times Y, r_X \oplus r_Y)}(\mu_1 \otimes \nu_1, \mu_2 \otimes \nu_2)
\le
\dpr^{(X, r_X)}(\mu_1,\mu_2) + \dpr^{(Y, r_Y)}(\nu_1,\nu_2).
\]
\end{lemma}

\begin{proof}
This is immediate from the observation that if
$\alpha$ and $\beta$ are probability measures on
$X \times X$ and $Y \times Y$, respectively, such that
\[
\alpha\{(x',x'') \in X \times X : r_X(x',x'') \ge \gamma\} \le \gamma
\]
and
\[
\beta\{(y',y'') \in Y \times Y : r_Y(y',y'') \ge \delta\} \le \delta
\]
for $\gamma, \delta > 0$, then
\[
\begin{split}
& \alpha\otimes \beta\{((x',y'),(x'',y'')) \in (X \times Y) \times (X \times Y) \\
& \quad : r_X(x',x'') + r_Y(y',y'') \ge \gamma + \delta\} \\
& \qquad \le \gamma + \delta, \\
\end{split}
\]
where, with a slight abuse of notation, we identify the measure
$\alpha \otimes \beta$ on $(X \times X) \times (Y \times Y)$
with its push-forward on $(X \times Y) \times (X \times Y)$
by the map $((x',x''),(y',y'')) \mapsto ((x',y'),(x'',y''))$.
\end{proof}

The next lemma is also probably well-known.

\begin{lemma}
\label{L:translation_invariance_Prohorov}
Suppose that $\mu_1$ and $\mu_2$ are two probability measures on a 
metric space $(X, r_X)$ and $\nu$ is a probability measure on
another metric space $(Y,r_Y)$.  Then,
\[
\dpr^{(X \times Y, r_X \oplus r_Y)}(\mu_1 \otimes \nu, \mu_2 \otimes \nu)
=
\dpr^{(X, r_X)}(\mu_1,\mu_2).
\]
\end{lemma}

\begin{proof}
It follows from Lemma~\ref{L:Prohorov_product} that
\[
\begin{split}
\dpr^{(X \times Y, r_X \oplus r_Y)}(\mu_1 \otimes \nu, \mu_2 \otimes \nu)
& \le
\dpr^{(X, r_X)}(\mu_1,\mu_2) + \dpr^{(Y, r_Y)}(\nu,\nu) \\
& =
\dpr^{(X, r_X)}(\mu_1,\mu_2). \\
\end{split}
\]

On the other hand, suppose that $\pi$ is a probability measure on 
$(X \times Y) \times (X \times Y)$ such that
$\pi(\cdot \times (X \times Y)) = \mu_1 \otimes \nu$, 
$\pi((X \times Y) \times \cdot) = \mu_2 \otimes \nu$
and
\[
\pi\{((x',y'),(x'',y'')) \in (X \times Y) \times (X \times Y) 
: r_X(x',x'') + r_Y(y',y'') \ge \epsilon\}
\le \epsilon
\]
for some $\epsilon > 0$. If $\rho$ is the push-forward of
$\pi$ by the map $((x',y'),(x'',y'')) \mapsto (x',x'')$, then it
is clear that $\rho(\cdot \times X) = \mu_1$, 
$\rho(X \times \cdot) = \mu_2$ and 
\[
\rho\{((x',x'') \in  X \times X
: r_X(x',x'')  \ge \epsilon\}
\le \epsilon,
\]
and hence 
\[
\dpr^{(X, r_X)}(\mu_1,\mu_2)
\le
\dpr^{(X \times Y, r_X \oplus r_Y)}(\mu_1 \otimes \nu, \mu_2 \otimes \nu).
\]
\end{proof}

%Recall that the {\em Hausdorff metric} on the set of
%closed subsets of a compact metric space $(Z, r_Z)$
%is given by
%\[
%\dhaus^{(Z, r_Z)}(F,G) 
%:= \inf
%\{\epsilon > 0 : F \subseteq G^\epsilon 
%\, \text{and} \,
% G \subseteq F^\epsilon\}.
% \]
% 
%The {\em Gromov-Hausdorff metric} is a metric on the space of 
%isometry classes of compact metric spaces. If $(X, r_X)$ and $(Y, r_Y)$ 
%are representatives of two isometry classes, the corresponding
%Gromov-Hausdorff distance is
%\[
%\dgh((X, r_X), (Y, r_Y)) 
%:= \inf_{(\phi_X ,\phi_Y ,Z)}
%\dhaus^{(Z, r_Z)}(\phi_X(X), \phi_Y(Y)),
%\]
%where the infimum is taken over all compact metric space $(Z, r_Z)$ and
%isometric embeddings $\phi_X$ of $X$
% and $\phi_Y$ of $Y$ into  $Z$. 

The {\em Gromov-Prohorov metric} is a metric on the space of 
equivalence classes of metric measure space (recall that
two metric measure spaces are equivalent if there is an isometry
mapping one to the other such that the probability measure on the
first is mapped to the probability measure on the second).  Given
two metric measure spaces $\cX = (X, r_X, \mu_X)$ and 
$\cY = (Y, r_Y, \mu_Y)$, the Gromov-Prohorov distance between their
equivalence classes is
\[
\dgpr(\cX,\cY) 
:= \inf_{(\phi_X, \phi_Y, Z)}
\dpr^{(Z, r_Z)}(\mu_X \circ \phi_X^{-1} ,   \mu_Y \circ \phi_Y^{-1}),
\]
where the infimum is taken over all metric spaces 
$(Z, r_Z)$ and isometric embeddings $\phi_X$ of $X$ and $\phi_Y$
of $Y$ into $Z$, and $\mu_X \circ \phi_X^{-1}$ (resp. 
$\mu_Y \circ \phi_Y^{-1}$)
denotes the push-forward of $\mu X$ by $\phi_X$ (resp. $\mu_Y$
by $\phi_Y$). It is easy to see that 
\begin{equation}
  \label{Eq:DGPr-to-neutral}
  \dgpr(\cX,\cE)=\inf_{x\in X} \inf\{\epsilon>0:\; 
  \mu_X\{y\in X:\; r_X(x,y)\geq \epsilon\}\leq\epsilon\}\,.
\end{equation}

\section{Inequalities for Laplace transforms}
\label{S:exponential_inequalities}

In this section we prove two inequalities about Laplace transforms of
nonnegative random variables that were used in the proof of 
Lemma~\ref{L:chi1_chiA}.

\begin{lemma}
\label{L:exponential_inequalities}
There are constants $\kappa', \kappa'' > 0$ such that
for any nonnegative random variable
$\xi$ we have 
\[
\kappa'((-\log(\bE[\exp(-\xi)])) \wedge 1)
\le 
\bE[\xi \wedge 1]
\le
\kappa'' ((-\log(\bE[\exp(-\xi)])) \wedge 1).
\]
\end{lemma}

\begin{proof}
Consider the first inequality.
Recall that
$1-\exp(-u) \le u$ for all $u \in \bR$.  
Furthermore, there is a constant $\gamma > 0$ such that
$1 - u \ge - \gamma \log u$ for $e^{-1} \le u \le 1$.
Thus,
\[
\begin{split}
\bE[\xi \wedge 1]
& \ge
1 - \bE[\exp(-(\xi \wedge 1))] \\
& \ge
\gamma(-\log(\bE[\exp(-(\xi \wedge 1))])). \\
\end{split}
\]
It will therefore suffice to show that there is a constant 
$\delta > 0$ such that
\[
\begin{split}
-\log(\bE[\exp(-(\xi \wedge 1))])
& \ge
\delta((-\log(\bE[\exp(-\xi)])) \wedge 1) \\
& = \delta(-\log(\bE[\exp(-\xi)] \vee e^{-1})) \\
\end{split}
\]
or, equivalently, that
\[
\bE[\exp(-(\xi \wedge 1))] 
\le 
(\bE[\exp(-\xi)] \vee e^{-1})^\delta
= 
\bE[\exp(-\xi)]^\delta \vee e^{- \delta}.
\]

That is, we need to show that we can choose $\delta > 0$ such that
if 
\begin{equation}
\label{E_exp_too_big}
\bE[\exp(-(\xi \wedge 1))] > e^{- \delta},
\end{equation}
then
\[
\bE[\exp(-(\xi \wedge 1))] \le \bE[\exp(-\xi)]^\delta.
\]
Moreover, since
\[
\bE[\exp(-(\xi \wedge 1))]
=
\bE[\exp(-\xi) \ind\{\xi < 1\}] +  e^{-1} \bP\{\xi \ge 1\}
\]
and
\[
\bE[\exp(-\xi)] \ge \bE[\exp(-\xi) \ind\{\xi < 1\}],
\]
it will be enough to establish that
\begin{equation}
\label{delta_power_bound}
1 + \frac{e^{-1} \bP\{\xi \ge 1\}}{\bE[\exp(-\xi) \ind\{\xi < 1\}]}
\le
\bE[\exp(-\xi) \ind\{\xi < 1\}]^{(\delta-1)}.
\end{equation}

Suppose that \eqref{E_exp_too_big} holds.
In that case
\[
\begin{split}
e^{- \delta} 
& < 
\bE[\exp(-(\xi \wedge 1))] \\
& =
\bE[\exp(-\xi) \ind\{\xi < 1\}] +  e^{-1} \bP\{\xi \ge 1\} \\
& \le
1 - \bP\{\xi \ge 1\} + e^{-1} \bP\{\xi \ge 1\},\\
\end{split}
\]
so that
\[
\bP\{\xi \ge 1\} 
< 
\frac{1 - e^{-\delta}}{1 - e^{-1}}
\]
and
\[
\bE[\exp(-\xi) \ind\{\xi < 1\}] 
> 
\frac{e^{- \delta} - e^{-1}}{1 - e^{-1}}.
\]
Therefore \eqref{delta_power_bound} will hold when
\begin{displaymath}
1 + \frac{e^{-1}(1 - e^{- \delta})}{e^{- \delta} - e^{-1}}
=
\frac{e^{- \delta}(1-e^{-1})}{e^{- \delta} - e^{-1}} 
\le 
\left(\frac{e^{- \delta} - e^{-1}}{1 - e^{-1}}\right)^{\delta-1} \\
\end{displaymath}
or, after some rearrangement, when
\[
e^{- \delta}  
\le \left(\frac{e^{- \delta} - e^{-1}}{1 - e^{-1}}\right)^\delta
\]
or, equivalently,
\[
e^{-1}
\le
\frac{e^{-\delta} - e^{-1}}{1 - e^{-1}}.
\]
This is certainly possible by taking
$\delta$ sufficiently small.  Numerically,
the upper bound on satisfactory values of $\delta$ given by
this argument is approximately $0.51012$.

Now consider the second inequality in the statement of the lemma.
Recall that $-\log(1-u) \ge u$ for $0 \le u < 1$.  Moreover,
there is a constant $0 < \beta < 1$ such that
$\exp(-u) \le 1 - \beta u$ for $0 \le u \le 1$.
We have
\[
\begin{split}
(- \log \bE[\exp(-\xi)]) \wedge 1
& =
- \log (\bE[\exp(-\xi)] \vee e^{-1}) \\
& \ge
- \log (\bE[\exp(-\xi)\vee e^{-1}]) \\
& =
- \log (\bE[\exp(-(\xi \wedge 1)]) \\
& \ge
1 - \bE[\exp(-(\xi \wedge 1)] \\
& \ge
\beta \bE[\xi \wedge 1], \\
\end{split}
\]
where we used Jensen's inequality for the first inequality. 
\end{proof}

\bigskip\noindent
{\bf Acknowledgments:} This work commenced while the authors were
attending a symposium the Institut Mittag-Leffler of the Royal Swedish
Academy of Sciences to honor the scientific work of Olav Kallenberg.
The authors are grateful to the referee for careful reading of the
manuscript, for drawing the authors' attention to the results from the
theory of Delphic semigroups and for the encouragement to extend the
setting from compact to general metric measure spaces.
%%% We thank Roger Purves for helpful comments.

\def\cprime{$'$}
\providecommand{\bysame}{\leavevmode\hbox to3em{\hrulefill}\thinspace}
\providecommand{\MR}{\relax\ifhmode\unskip\space\fi MR }
% \MRhref is called by the amsart/book/proc definition of \MR.
\providecommand{\MRhref}[2]{%
  \href{http://www.ams.org/mathscinet-getitem?mr=#1}{#2}
}
\providecommand{\href}[2]{#2}

\end{document}